\documentclass[11pt]{amsart}
\usepackage{amssymb,latexsym,graphicx,amscd}
\usepackage{lscape}
\usepackage[all]{xy}
\usepackage{color}
\usepackage{epic,eepic}
\usepackage{xspace}
\usepackage{mathrsfs}
\usepackage{setspace}
\newtheorem{Thm}{Theorem}[section]
\newtheorem{Lem}[Thm]{Lemma}

\newtheorem{Cor}[Thm]{Corollary}
\newtheorem{Prop}[Thm]{Proposition}
\newtheorem{Conj}[Thm]{Conjecture}

\newtheorem{Problem}[Thm]{Problem}

\newcommand{\A}{\mathbb{A}}

\newcommand{\Z}{\mathbb{Z}}
\newcommand{\Q}{\mathbb{Q}}
\newcommand{\N}{\mathbb{N}}
\newcommand{\C}{\mathbb{C}}

\newcommand{\df}{\colon}

\newcommand{\cA}{{\mathcal A}}
\newcommand{\cB}{{\mathcal B}}
\newcommand{\cC}{{\mathcal C}}

\newcommand{\cM}{{\mathcal M}}
\newcommand{\cN}{{\mathcal N}}
\newcommand{\cO}{{\mathcal O}}
\newcommand{\cP}{{\mathcal P}}

\newcommand{\cS}{{\mathcal S}}
\newcommand{\cT}{{\mathcal T}}
\newcommand{\cU}{{\mathcal U}}

\newcommand{\m}{\mathfrak{m}}

\newcommand{\ba}{\mathbf{a}}
\newcommand{\bb}{\mathbf{b}}

\newcommand{\bd}{\mathbf{d}}
\newcommand{\be}{\mathbf{e}}
\newcommand{\bff}{\mathbf{f}}

\newcommand{\bp}{{\mathbf p}}
\newcommand{\bq}{{\mathbf q}}

\newcommand{\bv}{{\mathbf v}}
\newcommand{\bx}{{\mathbf x}}

\newcommand{\LL}{\Lambda}

\newcommand{\rk}{\operatorname{rank}}

\newcommand{\md}{\operatorname{mod}}
\newcommand{\Mod}{\operatorname{Mod}}
\newcommand{\nil}{\operatorname{nil}}
\newcommand{\rep}{\operatorname{rep}}
\newcommand{\decrep}{\operatorname{decrep}}

\newcommand{\add}{\operatorname{add}}

\newcommand{\dimv}{\underline{\dim}}

\newcommand{\soc}{\operatorname{soc}}

\newcommand{\inj}{{\rm inj}}

\newcommand{\Hom}{\operatorname{Hom}}
\newcommand{\Ext}{\operatorname{Ext}}
\newcommand{\End}{\operatorname{End}}
\newcommand{\gEnd}{\operatorname{end}}
\newcommand{\gHom}{\operatorname{hom}}
\newcommand{\gExt}{\operatorname{ext}}

\newcommand{\length}{\operatorname{length}}

\newcommand{\Ima}{\operatorname{Im}}
\newcommand{\Ker}{\operatorname{Ker}}
\newcommand{\Coker}{\operatorname{Coker}}

\newcommand{\irr}{\operatorname{Irr}}
\newcommand{\decirr}{\operatorname{decIrr}}

\newcommand{\bsm}{\begin{smallmatrix}}
\newcommand{\esm}{\end{smallmatrix}}

\newcommand{\bbm}{\begin{matrix}}
\newcommand{\ebm}{\end{matrix}}

\newcommand{\bbsm}{\left(\begin{smallmatrix}}
\newcommand{\besm}{\end{smallmatrix}\right)}

\newcommand{\bbbm}{\left(\begin{matrix}}
\newcommand{\bebm}{\end{matrix}\right)}

\newcommand{\GL}{\operatorname{GL}}
\newcommand{\Gr}{\operatorname{Gr}}

\newcommand{\xra}{\xrightarrow}

\newcommand{\CQ}[1]{\C\langle\hspace{-0.05cm}\langle #1\rangle\hspace{-0.05cm}\rangle}

\newcommand{\up}{{\rm up}}
\newcommand{\sr}{{\rm s.r.}}
\newcommand{\alg}{{\rm alg}}
\newcommand{\dec}{{\rm dec}}

\begin{document}
%%%%%%%%%%%%%%%%
%%%%%%%%%%%%%%%%

%\today
\date{03.12.2012}

%%%%%%%%%%%%%%%%%%%%%%%%%%%%%%%%%%%%%%%%%%%%%
\title{Caldero-Chapoton algebras}
%%%%%%%%%%%%%%%%%%%%%%%%%%%%%%%%%%%%%%%%%%%%%

\author[G. Cerulli Irelli]{Giovanni Cerulli Irelli}
\address{Giovanni Cerulli Irelli\newline
Mathematisches Institut\newline
Universit\"at Bonn\newline
Endenicher Allee 60\newline
53115 Bonn\newline
Germany}
\email{cerulli@math.uni-bonn.de}

\author[D. Labardini-Fragoso]{Daniel Labardini-Fragoso}
\address{Daniel Labardini-Fragoso\newline
Mathematisches Institut\newline
Universit\"at Bonn\newline
Endenicher Allee 60\newline
53115 Bonn\newline
Germany}
\email{labardini@math.uni-bonn.de}

\author[J. Schr\"oer]{Jan Schr\"oer}
\address{Jan Schr\"oer\newline
Mathematisches Institut\newline
Universit\"at Bonn\newline
Endenicher Allee 60\newline
53115 Bonn\newline
Germany}
\email{schroer@math.uni-bonn.de}


%%%%%%%%%%%%%%%%%%%%%%%%%%%%%%%%%%%%%%%%%%%%%%%%%%%%%%%%%%%%%%%%%%%%%%%
%%%%%%%%%%%%%%%%%%%%%%%%%%%%%%%%%%%%%%%%%%%%%%%%%%%%%%%%%%%%%%%%%%%%%%%

\begin{abstract}
Motivated by the 
representation theory of quivers with potential introduced by Derksen, Weyman and Zelevinsky and by work of 
Caldero and Chapoton, who gave explicit formulae for the 
cluster variables of cluster algebras of Dynkin type,
we associate a \emph{Caldero-Chapoton algebra} $\cA_\LL$
to any (possibly infinite dimensional) basic algebra $\LL$.
By definition, $\cA_\LL$ is (as a vector space) generated by the \emph{Caldero-Chapoton functions} $C_\LL(\cM)$ of the decorated representations $\cM$ of $\LL$.
If $\LL = \cP(Q,W)$ is the Jacobian algebra defined by a 
2-acyclic quiver $Q$ with non-degenerate potential $W$, then we have
$\cA_Q \subseteq \cA_\LL \subseteq \cA_Q^\up$, where
$\cA_Q$ and $\cA_Q^\up$ are the cluster algebra and
the upper cluster algebra associated to $Q$.
The set $\cB_\LL$ of generic Caldero-Chapoton functions
is parametrized by the strongly reduced components
of the varieties of representations of the Jacobian algebra $\cP(Q,W)$ and
was introduced by Geiss, Leclerc and Schr\"oer.
Plamondon parametrized the strongly reduced components
for finite-dimensional basic algebras.
We generalize this to arbitrary basic algebras.
Furthermore, we prove
a decomposition theorem for strongly reduced components.
We define $\cB_\LL$ for arbitrary $\LL$, and we conjecture that
$\cB_\LL$ is a basis of the Caldero-Chapoton algebra $\cA_\LL$.
Thanks to the decomposition theorem, all elements of $\cB_\LL$ can be seen as generalized cluster monomials.
As another application, we obtain a new proof for the sign-coherence of
$g$-vectors.
\end{abstract}

%%%%%%%%%%%%%%%%
\maketitle
%%%%%%%%%%%%%%%%

\setcounter{tocdepth}{1}
\tableofcontents

\parskip2mm
%\parindent0pt

%%%%%%%%%%%%%%%%%%%%%%%%%%%%%%%%%%%%%%%%%%%%%%%%%%%%%%%%%%%%%%%%%%%%%%%
%%%%%%%%%%%%%%%%%%%%%%%%%%%%%%%%%%%%%%%%%%%%%%%%%%%%%%%%%%%%%%%%%%%%%%%

%%%%%%%%%%%%%%%%%%%%%%%%%%%%%%%%%%%%
%%%%%%%%%%%%%%%%%%%%%%%%%%%%%%%%%%%%

\section{Introduction}

%%%%%%%%%%%%%%%%%%%%%%%%%%%%%%%%%%%%
%%%%%%%%%%%%%%%%%%%%%%%%%%%%%%%%%%%%

%%%%%%%%%%
\subsection{}
%%%%%%%%%%
Let $\cA_Q$ be the Fomin-Zelevinsky cluster
algebra \cite{FZ1,FZ2} associated to 
a finite 2-acyclic quiver $Q$.
By definition $\cA_Q$ 
is generated by an inductively defined
set of rational functions, called cluster variables.
The cluster variables are contained in the set $\cM_Q$
of cluster monomials, which are by definition certain
monomials in the cluster variables.

Now let $W$ be a non-degenerate potential for $Q$, and let
$\LL = \cP(Q,W)$ be the associated Jacobian algebra introduced by Derksen, Weyman and Zelevinsky
\cite{DWZ1,DWZ2}.
The category of decorated representations of $\LL$ is
denoted by $\decrep(\LL)$.
To any $\cM \in \decrep(\LL)$ one can associate a
Laurent polynomial
$C_\LL(\cM)$, the Caldero-Chapoton function of $\cM$.
It follows from 
\cite{DWZ1,DWZ2}
that the cluster monomials form a subset of the set $\cC_\LL$ of
Caldero-Chapoton functions.

%%%%%%%%%%%%%%%%%%%%%%%%%%%%%%
\subsection{The generic basis conjecture}
%%%%%%%%%%%%%%%%%%%%%%%%%%%%%%
One of the main problems in cluster algebra theory is 
to find a basis of $\cA_Q$ with favourable properties.
As an important requirement, this basis should contain the set $\cM_Q$ of cluster monomials in a natural way.

The concept of strongly reduced irreducible components
of varieties of decorated representations of a Jacobian algebra $\LL$ was introduced
in \cite{GLSChamber}.
To each strongly reduced component $Z$ one can associate a
generic Caldero-Chapoton function $C_\LL(Z)$, see 
Sections~\ref{defCC} and \ref{defgenericCC}.
It was conjectured in \cite{GLSChamber} that the set
$\cB_\LL$ of generic Caldero-Chapoton functions forms a $\C$-basis
of $\cA_Q$.
Using a non-degenerate potential defined by Labardini \cite{La1,La2},
Plamondon \cite{P2} 
found a counterexample and then conjectured that
$\cB_\LL$ is a basis of the upper cluster algebra $\cA_Q^\up$.
This conjecture should also be wrong in general. 
We replace it by yet another conjecture. 

We study the \emph{Caldero-Chapoton algebra} 
$$
\cA_\LL := \langle C_\LL(\cM) \mid \cM \in \decrep(\LL) \rangle_\alg
$$
generated by all Caldero-Chapoton functions.
We do not restrict ourselves to Jacobian algebras, but
work with algebras $\LL$ defined as arbitrary quotients
of completed path algebras.
In particular, we generalize the notion of a Caldero-Chapoton
function to this general setup.
One easily checks that the functions $C_\LL(\cM)$ do not only generate
$\cA_\LL$ as an algebra but also as a vector space over
the ground field $\C$.

\begin{Conj}
$\cB_\LL$ is a $\C$-basis of $\cA_\LL$.
\end{Conj}

We show that the set
$\cB_\LL$ of generic Caldero-Chapoton functions is linearly independent provided
the kernel of the skew-symmetric incidence matrix $B_Q$ of $Q$ does not contain any non-zero element in $\Q_{\ge 0}^n$.
This generalizes \cite[Proposition~3.19]{P2}.

For $\LL = \cP(Q,W)$ a Jacobian algebra associated to a quiver
$Q$ with non-degenerate potential $W$ we have
$$
\cA_Q \subseteq \cA_\LL \subseteq \cA_Q^\up
$$
where $\cA_Q$ is the cluster algebra and $\cA_Q^\up$ is
the upper cluster algebra associated to $Q$.
(We refer to \cite{BFZ,DWZ1,FZ1} for missing definitions.)
For this special case, we give a list of open problems, which hopefully will lead
to a better understanding of the rather mysterious relation
between $\cA_Q$ and $\cA_Q^\up$.

%%%%%%%%%%%%%%%%%%%%%%%%%%%%%%%%%%%%%%%%%%%%%
\subsection{Parametrization of strongly reduced components}
%%%%%%%%%%%%%%%%%%%%%%%%%%%%%%%%%%%%%%%%%%%%%
Plamondon \cite[Theorem~1.2]{P2} parametrized the strongly reduced components
for finite-dimensional basic algebras. 
(For our (non-standard) definition of a basic algebra we refer to Section~\ref{sec2.1}.)
We generalize Plamondon's result 
to arbitrary basic algebras.
Let $\LL = \CQ{Q}/I$ be a basic algebra, where the
quiver $Q$ has $n$ vertices.
Let $\decirr(\LL)$ be the set of irreducible
components of all varieties $\decrep_{\bd,\bv}(\LL)$ of
decorated representations of $\LL$, where $(\bd,\bv)$ runs
through $\N^n \times \N^n$.
By $\decirr^\sr(\LL)$ we denote the subset of strongly
reduced components. (The definition is in Section~\ref{sec5}.)
Recall that $\decirr^\sr(\LL)$ parametrizes the elements in $\cB_\LL$.

Let
$$
G_\LL^\sr\df \decirr^\sr(\LL) \to \Z^n
$$
be the map sending $Z \in \decirr^\sr(\LL)$
to the generic $g$-vector $g_\LL(Z)$ of $Z$.
(The definition of a $g$-vector is in Section~\ref{sec3}.)
Using Plamondon's result for finite-dimensional
algebras, and a long-path truncation argument,
we get the following parametrization of
strongly reduced components for arbitrary $\LL$.

\begin{Thm}\label{paraintro}
For a basic algebra $\LL = \CQ{Q}/I$ the following hold:
\begin{itemize}

\item[(i)]
The map
$$
G_\LL^\sr\df \decirr^\sr(\LL) \to \Z^n
$$
is injective.

\item[(ii)]
The following are equivalent:
\begin{itemize}

\item[(a)]
$G_\LL^\sr$ is surjective.

\item[(b)]
$\overline{\LL} := \CQ{Q}/\overline{I}$ is finite-dimensional,
where $\overline{I}$ is the $\m$-adic closure of $I$.

\end{itemize}

\end{itemize}
\end{Thm}

%%%%%%%%%%%%%%%%%%%%%%%%%%%%%%%%%%%%%%%%%%%
\subsection{A decomposition theorem for strongly reduced components}
%%%%%%%%%%%%%%%%%%%%%%%%%%%%%%%%%%%%%%%%%%%
The notion of a direct sum
of irreducible components of representation varieties was introduced
in \cite{CBS}.
The Zariski closure $Z := \overline{Z_1 \oplus \cdots \oplus Z_t}$
of a direct sum of irreducible components $Z_1,\ldots,Z_t$
of varieties of representations of $\LL$
is always irreducible, but in general $Z$ is not an irreducible
component.
It was shown in
\cite{CBS} that $Z$ is an irreducible component
provided the dimension of the first extension group between
the components is generically zero.
The following decomposition theorem is an analogue for strongly reduced components.
Instead of extension groups, we work with
a generalization $E_\LL(-,?)$ of the Derksen-Weyman-Zelevinsky $E$-invariant \cite{DWZ2}. (We define $E_\LL(-,?)$ in Section~\ref{sec3}.)

\begin{Thm}\label{decomp2intro}
For $Z_1,\ldots,Z_t \in \decirr(\LL)$ the following are equivalent:
\begin{itemize}

\item[(i)]
$\overline{Z_1 \oplus \cdots \oplus Z_t}$ is a strongly reduced irreducible component.

\item[(ii)]
Each $Z_i$ is strongly reduced and
$E_\LL(Z_i,Z_j) = 0$ for all $i \not= j$.

\end{itemize}
\end{Thm}

Based on Theorem~\ref{decomp2intro}, we show that all 
elements of $\cB_\LL$ can be seen as $CC$-cluster monomials. (The $CC$-cluster monomials generalize 
Fomin and Zelevinsky's notion of cluster monomials.)

%%%%%%%%%%%%%%%%%%%%%%%%%%%%%%%%%
\subsection{Sign-coherence of $g$-vectors}
%%%%%%%%%%%%%%%%%%%%%%%%%%%%%%%%%
A subset $U$ of $\Z^n$ is
called \emph{sign-coherent} if for each $1 \le i \le n$
we have either $a_i \ge 0$ for all $(a_1,\ldots,a_n) \in U$,
or we have $a_i \le 0$ for all $(a_1,\ldots,a_n) \in U$.

The following theorem generalizes
\cite[Theorem~3.7(1)]{P2}.

\begin{Thm}\label{signcohintro}
Let $\LL$ be a basic algebra, and
let $Z_1,\ldots,Z_t \in \decirr^\sr(\LL)$ be strongly reduced
components. 
Assume that 
$$
\overline{Z_1 \oplus \cdots \oplus Z_t}
$$
is a strongly reduced component.
Then 
$\{ g_\LL(Z_1),\ldots,g_\LL(Z_t) \}$ 
is
sign-coherent.
\end{Thm}

%%%%%%%%%%
\subsection{}
%%%%%%%%%%
The paper is organized as follows.
In Section~\ref{sec2} we recall definitions and basic properties
of basic algebras and their (decorated) representations.
We also introduce truncations of basic algebras, which play a crucial
role in some of our proofs.
In Section~\ref{sec3} we introduce and study $g$-vectors
and $E$-invariants of decorated representations. 
Caldero-Chapoton functions and
Caldero-Chapoton algebras are defined in Section~\ref{sec4}.
Our main results Theorem~\ref{paraintro} and \ref{decomp2intro}
are proved in Section~\ref{sec5}.
In Section~\ref{sec6} we introduce component graphs, 
component clusters and $CC$-clusters, and we show that the
cardinality of loop-complete subgraphs of a component graph
is bounded by the number of simple modules. 
Section~\ref{sec7} explains the relation between Caldero-Chapoton algebras and cluster algebras.
Section~\ref{sec8} contains the proof of 
Theorem~\ref{signcohintro}.
Finally, in Section~\ref{sec9} we discuss several examples
of Caldero-Chapoton algebras.

%%%%%%%%%%%%%%%%
\subsection{Notation}
%%%%%%%%%%%%%%%%
We denote 
the composition of maps $f\colon M \to N$ and
$g\colon N \to L$ by $gf = g \circ f\colon M \to L$.
We write $|U|$ for the cardinality of a set $U$. 

A finite-dimensional module $M$ is \emph{basic} provided
it is a direct sum of pairwise non-isomorphic indecomposable
modules.
For a module $M$ and some $m \ge 1$ let $M^m$
be the direct sum of $m$ copies of $M$.

For a finite-dimensional algebra $\LL$ let
$\tau_\LL$ be its
Auslander-Reiten translation.
For an introduction to Auslander-Reiten theory we refer to the
books \cite{ARS} and \cite{ASS}.

For $n \ge 1$ and a set $S$, depending on the situation and
if no misunderstanding can occur,
we identify $S^n$ with the set of $(n \times 1)$- or
$(1 \times n)$-matrices with entries in $S$.
By $\N$ we denote the natural numbers, including zero.
For 
$\bd = (d_1,\ldots,d_n) \in \N^n$ let $|\bd| := d_1 + \cdots + d_n$.
For $n \in \N$ let $M_n(\Z)$ be the set of 
$(n \times n)$-matrices with integer entries.

For a ring $R$ let
$R[x_1^\pm,\ldots,x_n^\pm]$ be the algebra
of Laurent polynomials over $R$ in $n$ independent variables
$x_1,\ldots,x_n$.
For $\ba = (a_1,\ldots,a_n) \in \Z^n$ 
set 
$\bx^\ba := x_1^{a_1} \cdots x_n^{a_n}$.

%%%%%%%%%%%%%%%%%%%%%%%%%%%%%%%%%%%%%%%%%%%%
%%%%%%%%%%%%%%%%%%%%%%%%%%%%%%%%%%%%%%%%%%%%

\section{Basic algebras and decorated representations}\label{sec2}

%%%%%%%%%%%%%%%%%%%%%%%%%%%%%%%%%%%%%%%%%%%%
%%%%%%%%%%%%%%%%%%%%%%%%%%%%%%%%%%%%%%%%%%%%

%%%%%%%%%%%%%%%%%%%%%%%%%%%%%%%%%%%%%%%%%
\subsection{Basic algebras and quiver representations}\label{sec2.1}
%%%%%%%%%%%%%%%%%%%%%%%%%%%%%%%%%%%%%%%%%
Throughout, let $\C$ be the field of complex numbers.
A \emph{quiver} is a quadruple $Q = (Q_0,Q_1,s,t)$, where
$Q_0$ is a finite set of \emph{vertices}, $Q_1$ is a finite
set of \emph{arrows}, and $s,t\colon Q_1 \to Q_0$ are
maps.
For each arrow $a \in Q_1$ we call $s(a)$ and $t(a)$ the
starting and terminal vertex of $a$, respectively.
If not mentioned otherwise, we always assume that 
$Q_0 = \{ 1,\ldots,n \}$.
Let $B_Q = (b_{ij}) \in M_n(\Z)$, where
$$
b_{ij} := |\{ a \in Q_1 \mid s(a) = j,\, t(a) = i \}| - 
|\{ a \in Q_1 \mid s(a) = i,\, t(a) = j\}|.
$$

A \emph{path} in $Q$ 
is a tuple $p = (a_m,\ldots,a_1)$ of
arrows $a_i \in Q_1$ such that $s(a_{i+1}) = t(a_i)$ for all
$1 \le i \le m-1$.
Then $\length(p) := m$ is the \emph{length} of $p$.
Additionally, for each vertex $i \in Q_0$ there is a path
$e_i$ of length $0$.
We often just write $a_m \cdots a_1$ instead of $(a_m,\ldots,a_1)$.

A path $p = (a_m,\ldots,a_1)$ of length $m \ge 1$ is a
\emph{cycle} in $Q$, or more precisely an
$m$-\emph{cycle} in $Q$, if $s(a_1) = t(a_m)$.
The quiver $Q$ is \emph{acyclic} if there are no cycles in $Q$,
and for $s \ge 1$ the quiver $Q$ is called 
$s$-\emph{acyclic} if there are no $m$-cycles for
$1 \le m \le s$.

A \emph{representation} of a quiver $Q = (Q_0,Q_1,s,t)$ is a 
tuple $M = (M_i,M_a)_{i \in Q_0,a \in Q_1}$, where each 
$M_i$ is
a finite-dimensional $\C$-vector space, and $M_a\colon M_{s(a)} \to M_{t(a)}$ is a $\C$-linear map for 
each arrow $a \in Q_1$.
We call $\dimv(M) := (\dim(M_1),\ldots,\dim(M_n))$ the
\emph{dimension vector} of $M$.
Let $\dim(M) := \dim(M_1) + \cdots + \dim(M_n)$ be the
\emph{dimension} of $M$. 
For a path $p = (a_m,\ldots,a_1)$ in $Q$ let 
$M_p := M_{a_m} \circ \cdots \circ M_{a_1}$.
The representation $M$ is called \emph{nilpotent} provided
there exists some $N > 0$ such that
$M_p = 0$ for all paths $p$ in $Q$ with 
$\length(p) > N$.

For $i \in Q_0$ let
$S_i := (M_i,M_a)_{i,a}$ be the representation of $Q$ with
$M_i = \C$, $M_j = 0$ for all $j \not= i$, and
$M_a = 0$ for all $a \in Q_1$.
For a nilpotent representation $M$ the $i$th entry
$\dim(M_i)$ of its dimension vector $\dimv(M)$ 
equals the
Jordan-H\"older multiplicity $[M:S_i]$ of $S_i$ in $M$.

For $m \in \N$ let $\C Q_m$ be a $\C$-vector space with a $\C$-basis labeled by  
the set $Q_m$ of paths of length $m$ in $Q$.
Note that $\C Q_m$ is finite-dimensional.
We do not distinguish between a path $p$ of length $m$ and the corresponding basis vector in $\C Q_m$.

The \emph{completed path algebra} of a quiver $Q$ is
denoted by $\CQ{Q}$.
As a $\C$-vector space we have
$$
\CQ{Q} = \prod_{m \ge 0} \C Q_m.
$$
We write the elements in $\CQ{Q}$ as infinite sums
$\sum_{m \ge 0} a_m$ with $a_m \in \C Q_m$.
The product in $\CQ{Q}$ is then defined as
$$
(\sum_{i \ge 0} a_i)(\sum_{j \ge 0} b_j) := 
\sum_{k \ge 0} \sum_{i+j = k}a_ib_j.
$$
A \emph{potential} of $Q$ is an element $W = \sum_{m \ge 1} w_m$ of $\CQ{Q}$, where each $w_m$ is a $\C$-linear
combination of $m$-cycles in $Q$.
By definition, $W=0$ is also a potential.
The definition of a \emph{non-degenerate} potential can
be found in \cite[Section~7]{DWZ1}.

The category $\md(\CQ{Q})$ of finite-dimensional left $\CQ{Q}$-modules
can be identified with the category $\nil(Q)$ of nilpotent representations
of $Q$.

By $\m$ we denote the \emph{arrow ideal} in $\CQ{Q}$,
which is generated by the arrows of $Q$.
Thus for $p \ge 0$ we have
$$
\m^p = \prod_{m \ge p} \C Q_m.
$$
An ideal $I$ of $\CQ{Q}$ is \emph{admissible} if
$I \subseteq \m^2$.
We call an algebra $\LL$ \emph{basic} if $\LL = \CQ{Q}/I$ for
some quiver $Q$ and some admissible ideal $I$ of $\CQ{Q}$.

A \emph{representation} of a basic algebra $\LL = \CQ{Q}/I$ is a nilpotent representation of $Q$, which is annihilated by the ideal $I$.
We identify the category $\rep(\LL)$ of representations of $\LL$ 
with the category $\md(\LL)$ of finite-dimensional left $\LL$-modules.
Up to isomorphism the simple representations of $\LL$ are
the 1-dimensional representations 
$S_1,\ldots,S_n$.

The category of all (possibly infinite dimensional) $\LL$-modules
is denoted by $\Mod(\LL)$, we consider $\rep(\LL)$ as a subcategory of $\Mod(\LL)$.

%%%%%%%%%%%%%%%%%%%%%%%%%%%%%%%%%%%%%
\subsection{Decorated representations of quivers}
%%%%%%%%%%%%%%%%%%%%%%%%%%%%%%%%%%%%%
Let $\LL = \CQ{Q}/I$ be a basic algebra.
Following \cite{DWZ1} (and in fact the earlier publication \cite{MRZ}),
a \emph{decorated representation} of $\LL$
is a pair $\cM = (M,V)$, where $M$ is 
a representation of $\LL$ and $V = (V_1,\ldots,V_n)$ is
a tuple of finite-dimensional $\C$-vector spaces.
Let $\dimv(V) := (\dim(V_1),\ldots,\dim(V_n))$
and $\dim(V) := \dim(V_1) + \cdots + \dim(V_n)$. 
We call $\dimv(\cM) := (\dimv(M),\dimv(V))$ the \emph{dimension vector} of $\cM$.

One defines morphisms and direct sums of decorated representations in
the obvious way.
Let $\decrep(\LL)$ be the category of decorated representations of $\LL$.

Let $\cM = (M,V) \in \decrep(\LL)$.
We write $M = 0$ if all $M_i$ are zero, and $V = 0$ if all $V_i$ are zero.
Furthermore, $\cM = 0$ if $M=0$ and $V=0$.

For $1 \le i \le n$ set $\cS_i := (S_i,0)$, and let
$\cS_i^- := (0,V)$, where
$V_i = \C$ and $V_j = 0$ for all $j \not= i$.
The representations $\cS_i^-$ are
the \emph{negative simple}
decorated representations of $\LL$.

%%%%%%%%%%%%%%%%%%%%%%%%%%%%%%
\subsection{Varieties of representations}
%%%%%%%%%%%%%%%%%%%%%%%%%%%%%%
For $\bd = (d_1,\ldots,d_n) \in \N^n$ let
$\rep_\bd(\LL)$ be the affine variety of representations
of $\LL$
with dimension vector $\bd$.
By definition the closed points of $\rep_\bd(\LL)$ are the representations $M = (M_i,M_a)_{i \in Q_0,a \in Q_1}$ of $\LL$
with $M_i = \C^{d_i}$ for all $i \in Q_0$.
One can regard $\rep_\bd(\LL)$ as
a Zariski closed subset of the affine space
$$
\rep_\bd(Q) := \prod_{a \in Q_1} 
\Hom_\C(\C^{d_{s(a)}},\C^{d_{t(a)}}).
$$
For $\bd = (d_1,\ldots,d_n)$ let
$G_\bd := \prod_{i =1}^n \GL(\C^{d_i})$.
The group $G_\bd$ acts on $\rep_\bd(\LL)$ by conjugation.
More precisely, 
for $g = (g_1,\ldots,g_n) \in G_\bd$ 
and $M \in \rep_\bd(\LL)$ let
$$
g.M := (M_i,g_{t(a)}^{-1}M_ag_{s(a)})_{i \in Q_0,a \in Q_1}.
$$
For $M \in \rep_\bd(\LL)$ let $\cO(M)$ be the $G_\bd$-orbit
of $M$. 
The $G_\bd$-orbits are in bijection with the isomorphism
classes of representations of $\LL$ with dimension vector $\bd$.

For $(\bd,\bv) \in \N^n \times \N^n$ let
$\decrep_{\bd,\bv}(\LL)$ be the affine
variety of decorated representations $\cM = (M,V)$
with $M \in \rep_\bd(\LL)$ and
$V = \C^\bv := (\C^{v_1},\ldots,\C^{v_n})$,
where $\bv = (v_1,\ldots,v_n)$.

For $\cM = (M,V) \in \decrep_{\bd,\bv}(\LL)$ define
$g.\cM := (g.M,V)$.
This defines a $G_\bd$-action on $\decrep_{\bd,\bv}(\LL)$.
The $G_\bd$-orbit of $\cM$ is denoted by $\cO(\cM)$.
We have
\begin{equation}\label{orbitdim}
\dim \cO(\cM) = \dim \cO(M) = \dim G_\bd - \dim\End_\LL(M),
\end{equation}
see for example \cite{G}.

%%%%%%%%%%%%%%%%%%%%%%%%%%%
\subsection{Quiver Grassmannians}
%%%%%%%%%%%%%%%%%%%%%%%%%%%
Let $\LL = \CQ{Q}/I$ be a basic algebra.
For a representation $M = (M_i,M_a)_{i \in Q_0, a \in Q_1}$ 
of $\LL$ and $\be \in \N^n$
let $\Gr_\be(M)$ be the \emph{quiver Grassmannian}
of subrepresentations $U$ of $M$ with $\dimv(U) = \be$.
(By definition a \emph{subrepresentation} of $M$ is a tuple 
$U = (U_i)_{i \in Q_0}$ of subspaces $U_i \subseteq M_i$ such that
$M_a(U_{s(a)}) \subseteq U_{t(a)}$
for all $a \in Q_1$.)
So $\Gr_\be(M)$ is a projective variety, which can be seen
as a closed subvariety of the product of the classical Grassmannians $\Gr_{e_i}(M_i)$ of $e_i$-dimensional
subspaces of $M_i$, where $\be = (e_1,\ldots,e_n)$.
Let $\chi(\Gr_\be(M))$ be the Euler-Poincar\'e 
characteristic of $\Gr_\be(M)$.

%%%%%%%%%%%%%%%%%%%%%%%%%%%%%%%%
\subsection{Truncations of basic algebras}\label{sectiontruncation}
%%%%%%%%%%%%%%%%%%%%%%%%%%%%%%%%
For a basic algebra $\LL = \CQ{Q}/I$ and some
$p \ge 2$ let
$$
\LL_p := \CQ{Q}/(I+\m^p)
$$
be 
the $p$-\emph{truncation} of $\LL$.
Clearly, $\LL_p$ is a finite-dimensional basic algebra. 
Let
$$
\overline{\LL} := \CQ{Q}/\overline{I}
$$
where 
$$
\overline{I} := \bigcap_{p \ge 0} (I+\m^p)
$$
is the $\m$-adic closure of $I$ in $\CQ{Q}$.
We obtain the
following commutative diagram with exact rows, where all
morphisms, whose label contains the symbol $\iota$ (resp. $\pi$) 
are the obvious canonical monomorphisms (resp. epimorphisms).
$$
\xymatrix{
I \ar[r]^{\iota}\ar[d]_{\overline{\iota}} & \CQ{Q} \ar[r]^{\pi}\ar@{=}[d] & \LL \ar[d]^{\overline{\pi}}\\
\overline{I} \ar[r]^{\iota_\infty}\ar@/_4pc/[dd]_{\iota_{\infty,p}} & \CQ{Q} \ar[r]^{\pi_\infty}
& \overline{\LL}\ar@/^4pc/[dd]^{\pi_{\infty,p}}\\
\vdots \ar[d]_{\iota_{p+1,p}}& \vdots \ar@{=}[d]& \vdots \ar[d]^{\pi_{p+1,p}}\\
I+\m^p \ar[r]^{\iota_p}\ar[d]_{\iota_{p,p-1}} & \CQ{Q} \ar[r]^{\pi_p}\ar@{=}[d] & 
\LL_p \ar[d]^{\pi_{p,p-1}}\\
I+\m^{p-1} \ar[r]^{\iota_{p-1}}\ar[d]_{\iota_{p-1,p-2}} & \CQ{Q} \ar[r]^{\pi_{p-1}}\ar@{=}[d] & 
\LL_{p-1} \ar[d]^{\pi_{p-1,p-2}}\\
\vdots \ar[d]_{\iota_{3,2}}& \vdots \ar@{=}[d]& \vdots \ar[d]^{\pi_{3,2}} \\
I+\m^2 \ar[r]^{\iota_2} & \CQ{Q} \ar[r]^{\pi_2} & 
\LL_2
}
$$
We have
$$
\overline{\LL} =
\underleftarrow{\lim}(\LL_p),
$$
i.e.
the algebra $\overline{\LL}$ is the inverse limit of
the inverse system 
$$
(\cdots \to \LL_p \to \cdots \to \LL_3 \to \LL_2).
$$
The epimorphisms in the third column of the above diagram induce  sequences
$$
\rep(\LL_2) \to \cdots \to \rep(\LL_p) \to \cdots 
\rep(\overline{\LL}) \to \rep(\LL)
$$
and
$$
\decrep(\LL_2) \to \cdots \to \decrep(\LL_p) \to \cdots 
\decrep(\overline{\LL}) \to \decrep(\LL)
$$
of embedding functors.
We can consider these as chains of subcategories of
$\rep(\LL)$ and $\decrep(\LL)$, respectively.
The following lemma is straightforward.

\begin{Lem}
For any basic algebra $\LL$ we have
$$
\rep(\LL) = \rep(\overline{\LL}) =
\bigcup_{p \ge 2} \rep(\LL_p)
\text{\;\;\; and \;\;\;}
\decrep(\LL) = \decrep(\overline{\LL}) =
\bigcup_{p \ge 2} \decrep(\LL_p).
$$
\end{Lem}

For $1 \le i \le n$ and $p \ge 2$ let $I_{i,p} \in \rep(\LL_p)$,
$\overline{I}_i \in {\rm Mod}(\overline{\LL})$ and 
$I_i \in {\rm Mod}(\LL)$ be the
injective envelopes of the simple module $S_i$.
The above embedding functors yield a chain
$$
I_{i,2} \subseteq I_{i,3} \subseteq \cdots \subseteq I_{i,p} \subseteq \cdots \overline{I}_i
$$
of submodules of $\overline{I}_i$, and we have
$$
\overline{I}_i = \bigcup_{p \ge 2} I_{i,p}.
$$

\begin{Lem}\label{trunclemma}
Let $\LL = \CQ{Q}/I$ be a basic algebra.
Then the following hold:
\begin{itemize}

\item[(i)]
Let $\cM = (M,V) \in \decrep(\LL)$.
If $p \ge \dim(M)$, then $\cM$ is in the image of the embedding $\decrep(\LL_p) \to \decrep(\LL)$.

\item[(ii)]
Let $M,N \in \rep(\LL)$.
If $p  \ge \dim(M),\dim(N)$, then 
$$
\dim \Hom_{\LL_p}(M,N) = \dim \Hom_\LL(M,N).
$$

\item[(iii)]
Let $M,N \in \rep(\LL)$.
If $p  \ge \dim(M)+\dim(N)$, then 
$$
\dim \Ext_{\LL_p}^1(M,N) = \dim \Ext_\LL^1(M,N).
$$

\item[(iv)]
Let $(\bd,\bv)  \in \N^n \times \N^n$.
If $p \ge |\bd|$, then
$\decrep_{\bd,\bv}(\LL_p) = \decrep_{\bd,\bv}(\LL)$.

\end{itemize}
\end{Lem}

\begin{proof}
Let $a_m \cdots a_1$ be a path of length $m$ in $Q$,
and let $M$ be a
representation of $\LL$.
We can see $M$ as a nilpotent representation of $Q$.
For any non-zero vector $v_0 \in M$ set 
$v_i := a_i \cdots a_1v_0$ for $1 \le i \le m$.
Assume that each of the vectors $v_1,\ldots,v_m$ is non-zero.
We claim that $v_0,v_1,\ldots,v_m$ are pairwise different and
linearly independent.
Let  
$b$ be a path of maximal length such that $bv_0 \not= 0$.
Such a path $b$ exists, because $M$ is nilpotent.
By induction $v_1,\ldots,v_m$ are linearly independent.
Assume now that
$$
v_0 = \sum_{i=1}^m \lambda_i v_i 
$$
for some $\lambda_i \in \C$.
We have $v_i = a_i \cdots a_1v_0$.
Therefore we get
$$
bv_0 = \sum_{i=1}^m \lambda_i ba_i \cdots a_1v_0.
$$
Since $b a_i\cdots a_1$ is either zero or a path of length
$i+\length(b)$, we have $ba_i \cdots a_1v_0 = 0$
for all $1 \le i \le m$.
Since $bv_0 \not= 0$, this is a contradiction.
Therefore $v_0,v_1,\ldots,v_m$ are linearly independent.
It follows that for any $\cM \in \decrep(\LL)$ with
$\dimv(\cM) = (\bd,\bv)$ 
and any path $b$ with 
$\length(b) \ge |\bd|$ we have $b\cM = 0$.
This implies (i). 
Parts (ii) and (iv) are easy consequences of (i).
Any extension of representations $M$ and $N$ of $\LL$
is a representation of $\LL$ of dimension $\dim(M) + \dim(N)$.
This implies (iii).
\end{proof}

%%%%%%%%%%%%%%%%%%%%%%%%%%%%%%%%%%%%%%%%%%%%
%%%%%%%%%%%%%%%%%%%%%%%%%%%%%%%%%%%%%%%%%%%%

\section{$E$-invariants and $g$-vectors of decorated
representations}\label{sec3}

%%%%%%%%%%%%%%%%%%%%%%%%%%%%%%%%%%%%%%%%%%%%
%%%%%%%%%%%%%%%%%%%%%%%%%%%%%%%%%%%%%%%%%%%%

%%%%%%%%%%%%%%%%%%%%%%%%%%%%%%%%%%
\subsection{Definition of $E$-invariants and $g$-vectors}\label{sec3.1}
%%%%%%%%%%%%%%%%%%%%%%%%%%%%%%%%%%
Let $Q$ be a quiver, and let $W$ be a potential of $Q$. 
Let $\LL = \cP(Q,W)$ be the associated
Jacobian algebra \cite[Section~3]{DWZ1}.
For decorated representations $\cM$ and $\cN$ of $\LL$
the $g$-\emph{vector} $g(\cM)$ and the invariants
$E^\inj(\cM)$ and $E^\inj(\cM,\cN)$
were defined in \cite{DWZ2}, where
$E^\inj(\cM)$ is called the $E$-\emph{invariant}
of $\cM$.
We define invariants $g_\LL(\cM)$, $E_\LL(\cM)$ and
$E_\LL(\cM,\cN)$ of decorated representations $\cM$ and $\cN$ of an arbitrary basic algebra $\LL = \CQ{Q}/I$ as follows.

For a decorated representation $\cM = (M,V)$ of
$\LL$ let
$$
g_\LL(\cM) := (g_1,\ldots,g_n)
$$
with 
$$
g_i := g_i(\cM) := 
-\dim \Hom_\LL(S_i,M) + \dim \Ext_\LL^1(S_i,M) + \dim(V_i)
$$
be the $g$-\emph{vector} of $\cM$.

For decorated representations $\cM = (M,V)$ and
$\cN = (N,W)$ of $\LL$ let
$$
E_\LL(\cM,\cN) := \dim \Hom_\LL(M,N) + \sum_{i=1}^n
\dim(M_i)g_i(\cN).
$$
The $E$-\emph{invariant} of $\cM$ is defined as
$E_\LL(\cM) := E_\LL(\cM,\cM)$.

\begin{Lem}
Let $\LL = \cP(Q,W)$, where $W$ is a potential
of $Q$.
For $\cM,\cN \in \decrep(\LL)$ the
following hold:
\begin{itemize}

\item[(i)]
$g_\LL(\cM) = g(\cM)$.

\item[(ii)]
$E_\LL(\cM,\cN) = E^\inj(\cM,\cN)$.

\end{itemize}
\end{Lem}

\begin{proof}
Part (i) follows from \cite[Lemma~4.7, Proposition~4.8]{P1}.
It can also be shown in a more elementary way by using
the exact sequence displayed in
\cite[Equation~(10.4)]{DWZ2}.
Part (ii) is a direct consequence of (i) and the definition of
$E_\LL(\cM,\cN)$ and
$E^\inj(\cM,\cN)$.
\end{proof}

%%%%%%%%%%%%%%%%%%%%%%%%%%%%%%%%%%%%%%%%%%%%%
\subsection{Homological interpretation of the $E$-invariant}
%%%%%%%%%%%%%%%%%%%%%%%%%%%%%%%%%%%%%%%%%%%%%
For $1 \le i \le n$ let $I_i \in \Mod(\LL)$ be the injective envelope of the
simple representation $S_i$ of $\LL$.
One easily checks that the socle $\soc(I_i)$ of $I_i$ is
isomorphic to $S_i$, and that
\begin{equation}\label{dimMi}
\dim \Hom_\LL(M,I_i) = \dim(M_i)
\end{equation}
for all $M \in \rep(\LL)$.
Note that 
in general $I_i$ is infinite dimensional.
For $M \in \rep(\LL)$ let
$$
0 \to M \xra{f} I_0^\LL(M) \to I_1^\LL(M)
$$
denote a minimal injective presentation of $M$.
The modules $I_0^\LL(M)$ and $I_1^\LL(M)$ are 
up to isomorphism uniquely determined by $M$.

We will need the following theorem due to Auslander and Reiten.

\begin{Thm}[{\cite[Theorem~1.4 (b)]{AR}}]\label{AR}
Let $M$ and $N$ be representations of a finite-dimensional basic
algebra $\LL$.
Then we have
\begin{multline*}
\dim \Hom_\LL(\tau_\LL^-(N),M) =
\dim \Hom_\LL(M,N) - \dim \Hom_\LL(M,I_0^\LL(N)) \\ 
+ \dim \Hom_\LL(M,I_1^\LL(N)).
\end{multline*}
\end{Thm}

\begin{Lem}\label{HomologicalE1}
Let $\LL = \CQ{Q}/I$ be a finite-dimensional basic algebra, and let $M \in \rep(\LL)$.
Let 
$$
0 \to M \xra{f} I_0^\LL(M) \to I_1^\LL(M)
$$ 
be a minimal injective
presentation of $M$.
Then for $1 \le i \le n$ we have
\begin{itemize}

\item[(i)]
$[\soc(I_0^\LL(M)):S_i] = [\soc(M):S_i] = \dim \Hom_\LL(S_i,M)$ and
$$
I_0^\LL(M) \cong I_1^{\dim \Hom_\LL(S_1,M)} \oplus \cdots \oplus
I_n^{\dim \Hom_\LL(S_n,M)}.
$$

\item[(ii)]
$[\soc(I_1^\LL(M)):S_i] = [\soc(\Coker(f)):S_i] = 
\dim \Ext_\LL^1(S_i,M)$ and
$$
I_1^\LL(M) \cong I_1^{\dim \Ext_\LL^1(S_1,M)} \oplus \cdots \oplus
I_n^{\dim \Ext_\LL^1(S_n,M)}.
$$

\end{itemize}
\end{Lem}

\begin{proof}
Since $I_0^\LL(M)$ is the injective envelope  of $M$,
we have $\soc(M) \cong \soc(I_0^\LL(M))$.
This implies (i).
By the construction of minimal injective presentations,
$I_1^\LL(M)$ is the injective envelope  of $\Coker(f)$. 
It follows that $\soc(\Coker(f)) \cong \soc(I_1^\LL(M))$.
We apply the functor $\Hom_\LL(S_i,-)$ to 
the exact sequence
$$
0 \to M \xra{f} I_0^\LL(M) \to \Coker(f) \to 0.
$$
This yields an exact sequence
\begin{multline*}
0 \to \Hom_\LL(S_i,M) \xra{F} \Hom_\LL(S_i,I_0^\LL(M)) \to
\Hom_\LL(S_i,\Coker(f))
\xra{G} 
\Ext_\LL^1(S_i,M) \to 0
\end{multline*}
Here we used that
$I_0^\LL(M)$ is injective, which implies
$\Ext_\LL^1(S_i,I_0^\LL(M)) = 0$.
By (i) we know that $F$ is an isomorphism. 
Thus $G$ is also an isomorphism.
This implies (ii).
\end{proof}

Combinining Lemma~\ref{trunclemma} and Lemma~\ref{HomologicalE1}
yields the following result.

\begin{Lem}\label{lemma3.3}
Let $\cM = (M,V)$ be a decorated representation of a basic algebra $\LL$, and let
$g_\LL(\cM) = (g_1,\ldots,g_n)$ be the $g$-vector of $\cM$.
If $p > \dim(M)$, then
$$
g_i = -[\soc(I_0^{\LL_p}(M)):S_i] + [\soc(I_1^{\LL_p}(M)):S_i] + \dim(V_i)
$$
for all $1 \le i \le n$.
\end{Lem}

The following result is a homological interpretation
of the $E$-invariant in terms of Auslander-Reiten translations.
This can be seen as a generalization of \cite[Corollary~10.9]{DWZ2}.

\begin{Prop}\label{HomologicalE2}
Let $\cM = (M,V)$ and $\cN = (N,W)$ be decorated representations of a basic algebra
$\LL$.
If $p > \dim(M),\dim(N)$, then
$$
E_\LL(\cM,\cN) = E_{\LL_p}(\cM,\cN) = \dim \Hom_{\LL_p}(\tau_{\LL_p}^-(N),M) 
+  \sum_{i=1}^n\dim(M_i)\dim(W_i).
$$
In particular, we have
$$
E_{\LL_p}(\cM,\cN) = E_{\LL_q}(\cM,\cN)
$$
and
$$
\dim \Hom_{\LL_p}(\tau_{\LL_p}^-(N),M) = \dim \Hom_{\LL_q}(\tau_{\LL_q}^-(N),M)
$$
for all $p,q > \dim(M),\dim(N)$.
\end{Prop}

\begin{proof}
Since $p > \dim(M),\dim(N)$ we can apply
Lemma~\ref{trunclemma} and get
\begin{align*}
\dim \Hom_{\LL_p}(M,N) &= \dim \Hom_\LL(M,N),\\
\dim \Hom_{\LL_p}(S_i,N) &= \dim \Hom_\LL(S_i,N),\\
\dim \Ext_{\LL_p}^1(S_i,N) &= \dim \Ext_\LL^1(S_i,N).
\end{align*}
Let
$$
0 \to N \to I_0^{\LL_p}(N) \to I_1^{\LL_p}(N)
$$
be a minimal injective presentation
of $N$, where we regard $N$ now as a representation
of $\LL_p$.
It follows from
Lemma~\ref{HomologicalE1} and Equation~(\ref{dimMi})
that
\begin{align*}
\dim \Hom_{\LL_p}(M,I_0^{\LL_p}(N)) &= \sum_{i=1}^n \dim(M_i)\dim \Hom_{\LL_p}(S_i,N),\\
\dim \Hom_{\LL_p}(M,I_1^{\LL_p}(N)) &= \sum_{i=1}^n \dim(M_i)\dim \Ext_{\LL_p}^1(S_i,N).
\end{align*}
This implies
\begin{align*}
E_\LL(\cM,\cN) &= 
\dim \Hom_\LL(M,N) + 
\sum_{i=1}^n  \dim(M_i)(-\dim \Hom_\LL(S_i,N) + \dim \Ext_\LL^1(S_i,N))\\ 
&\;\;\;\; + 
\sum_{i=1}^n \dim(M_i)\dim(W_i)\\
&= 
\dim \Hom_{\LL_p}(M,N) + 
\sum_{i=1}^n  \dim(M_i)(-\dim \Hom_{\LL_p}(S_i,N) + \dim \Ext_{\LL_p}^1(S_i,N))\\ 
&\;\;\;\; + 
\sum_{i=1}^n \dim(M_i)\dim(W_i)\\
&= \dim \Hom_{\LL_p}(M,N) - \dim \Hom_{\LL_p}(M,I_0^{\LL_p}(N))
+ \dim \Hom_{\LL_p}(M,I_1^{\LL_p}(N))\\
&\;\;\;\; +\sum_{i=1}^n \dim(M_i)\dim(W_i).
\end{align*}
The first equality follows from Lemmas~\ref{trunclemma}, 
\ref{HomologicalE1} and \ref{lemma3.3}.
The second equality says that
$E_\LL(\cM,\cN) = E_{\LL_p}(\cM,\cN)$.
Applying 
Theorem~\ref{AR} yields 
$$
E_{\LL_p}(\cM,\cN) = \dim \Hom_{\LL_p}(\tau_{\LL_p}^-(N),M) 
+  \sum_{i=1}^n\dim(M_i)\dim(W_i).
$$
This finishes the proof.
\end{proof}

\begin{Cor}
For decorated representations $\cM$ and $\cN$ of a basic
algebra $\LL$ we have
$$
E_\LL(\cM,\cN) \ge 0.
$$
\end{Cor}

%%%%%%%%%%%%%%%%%%%%%%%%%%%%%%%%%%%%
%%%%%%%%%%%%%%%%%%%%%%%%%%%%%%%%%%%%

\section{Caldero-Chapoton algebras}\label{sec4}

%%%%%%%%%%%%%%%%%%%%%%%%%%%%%%%%%%%%
%%%%%%%%%%%%%%%%%%%%%%%%%%%%%%%%%%%%

%%%%%%%%%%%%%%%%%%%%%%%%%%%%%%%
\subsection{Caldero-Chapoton functions}\label{defCC}
%%%%%%%%%%%%%%%%%%%%%%%%%%%%%%%
To any basic algebra $\LL = \CQ{Q}/I$
we associate a
set of Laurent polynomials in $n$ independent variables 
$x_1,\ldots,x_n$ as follows.
The \emph{Caldero-Chapoton function} 
associated to a decorated 
representation $\cM = (M,V)$ of $\LL$
is defined as
$$
C_\LL(\cM) :=
\bx^{g_\LL(\cM)} \sum_{\be \in \N^n} \chi(\Gr_\be(M)) 
\bx^{B_Q \be},
$$
where $B_Q$ and $g_\LL(\cM)$ are defined as in Sections~\ref{sec2.1}
and \ref{sec3.1}, respectively.
Note that $C_\LL(\cM) \in \Z[x_1^\pm,\ldots,x_n^\pm]$.
Let 
$$
\cC_\LL := \{ C_\LL(\cM) \mid \cM \in \decrep(\LL) \}
$$
be the set of Caldero-Chapoton functions associated to 
$\LL$.
For $\cM = (M,0)$ we sometimes write $C_\LL(M)$ instead of
$C_\LL(\cM)$.

The definition of $C_\LL(\cM)$ is motivated by 
the (different versions of) Caldero-Chapoton functions appearing in the theory of cluster algebras, see for example
\cite[Section~1]{Pa}.
Such functions first appeared in work of Caldero and Chapoton
\cite[Section~3]{CC}, where they show that the cluster variables
of a cluster algebra of a Dynkin quiver are
Caldero-Chapoton functions.

\begin{Lem}\label{additive}
For decorated representations $\cM = (M,V)$ and
$\cN = (N,W)$ the following hold:
\begin{itemize}

\item[(i)]
$g_\LL(\cM \oplus \cN) = g_\LL(\cM) + g_\LL(\cN)$.

\item[(ii)]
$C_\LL(\cM) = C_\LL(M,0)C_\LL(0,V)$.

\item[(iii)]
$C_\LL(\cM \oplus \cN) = C_\LL(\cM) C_\LL(\cN)$.

\end{itemize}
\end{Lem}

\begin{proof}
Part (i) follows directly from the definitions and from
the additivity of the functors $\Hom_\LL(-,?)$ and
$\Ext_\LL^1(-,?)$.
To prove (ii),
let $\cM = (M,V)$ be a decorated representation of $\LL$.
For the decorated representation $(0,V)$ we have
$$
C_\LL(0,V) = \prod_{i=1}^n x_i^{v_i}
$$
where $\dimv(V) = (v_1,\ldots,v_n)$.
For the decorated representation $(M,0)$
we have
$$
C_\LL(M,0) := \bx^{g_\LL(M,0)} \sum_{\be \in \N^n} \chi(\Gr_\be(M)) 
\bx^{B_Q \be}
$$
where $g_i(M,0) = -\dim \Hom_\LL(S_i,M) + 
\dim \Ext_\LL^1(S_i,M)$ for $1 \le i \le n$.
Now one easily checks that $C_\LL(\cM) = C_\LL(M,0)C_\LL(0,V)$.
Thus (ii) holds.
Now (iii) follows from (i), (ii) and the well known formula 
$$
\chi(\Gr_\be(M \oplus N)) =
\sum_{(\be',\be'')} 
\chi(\Gr_{\be'}(M)) \chi(\Gr_{\be''}(N))
$$
where the sum runs over all pairs $(\be',\be'') \in \N^n \times
\N^n$ such that $\be' + \be'' = \be$, see for example
\cite[Proof of Proposition~3.2]{DWZ2}.
\end{proof}

%%%%%%%%%%%%%%%%%%%%%%%%%%%%%%%%%%%%%%%%%
\subsection{Definition of a Caldero-Chapoton algebra}
%%%%%%%%%%%%%%%%%%%%%%%%%%%%%%%%%%%%%%%%%
In the previous section, we associated to a basic algebra 
$\LL$ the set
$$
\cC_\LL = \{ C_\LL(\cM) \mid \cM \in \decrep(\LL) \}
$$
of Caldero-Chapoton functions.
Clearly, $\cC_\LL$ is a subset of the integer Laurent
polynomial ring 
$\Z[x_1^\pm,\ldots,x_n^\pm]$ generated by the variables
$x_1,\ldots,x_n$.
By definition the \emph{Caldero-Chapoton algebra} 
$\cA_\LL$ 
associated to $\LL$ is the 
$\C$-subalgebra of 
$\C[x_1^\pm,\ldots,x_n^\pm]$
generated by $\cC_\LL$.
The following is a direct consequence of 
Lemma~\ref{additive}(iii).

\begin{Lem}\label{alggen}
The set $\cC_\LL$ generates $\cA_\LL$ as a $\C$-vector space.
\end{Lem}

In this generality, Caldero-Chapoton algebras might not
be so useful. (One could generalize even more 
by replacing the matrix $B_Q$ in the definition of the functions $C_\LL(\cM)$
by any other matrix in $M_n(\Z)$.)
But the case where $\LL$ is the Jacobian algebra
$\cP(Q,W)$ of a quiver $Q$ with non-degenerate potential
$W$ (see \cite{DWZ1} for missing definitions) should certainly
be of interest.
In this case, based on work of Palu \cite{Pa}, Plamondon \cite{P1} considered a version of Caldero-Chapoton
functions using the Amiot cluster category \cite{A}.
In contrast, we follow Derksen, Weyman and Zelevinsky's \cite{DWZ1,DWZ2} approach and define and study Caldero-Chapoton functions purely
in terms
of the representation theory of the Jacobian algebra without passing
to the cluster category.

%%%%%%%%%%%%%%%%%%%%%%%%%%%%%%%%%%%%%
\subsection{Linear independence of Caldero-Chapoton functions}
%%%%%%%%%%%%%%%%%%%%%%%%%%%%%%%%%%%%%
Let $\LL = \CQ{Q}/I$ be a basic algebra.
Except in some trivial cases, the set $\cC_\LL$ of Caldero-Chapoton functions associated
to decorated representations of $\LL$ is 
linearly dependent. 
Often the Caldero-Chapoton functions 
satisfy beautiful relations, which should
be studied more intensively. 
On the other hand, by Lemma~\ref{alggen}, there are $\C$-bases
of $\cA_\LL$
consisting only of Caldero-Chapoton functions.
Our aim is to provide a candidate $\cB_\LL$ for such a basis.
Before constructing $\cB_\LL$ in Section~\ref{sec5}, 
we prove the following
criterion for linear independence of certain sets of
Caldero-Chapoton functions.

Let 
\begin{align*}
\Q_{\ge 0}^n &:= \{ (a_1,\ldots,a_n) \in \Q^n \mid 
a_i \ge 0 \text{ for all } i \},\\
\Q_{> 0}^n &:= \{ (a_1,\ldots,a_n) \in \Q^n \mid 
a_i > 0 \text{ for all } i \}.
\end{align*}

\begin{Prop}\label{li2}
Let $\LL = \CQ{Q}/I$ be a basic algebra.
Let $\cM_j$, $j \in J$ be 
decorated representations of $\LL$.
Assume the following:
\begin{itemize}

\item[(i)]
$\Ker(B_Q) \cap \Q_{\ge 0}^n = 0$.

\item[(ii)]
The
$g$-vectors $g_\LL(\cM_j)$, $j \in J$ are
pairwise different.

\end{itemize}
Then the Caldero-Chapoton functions $C_\LL(\cM_j)$,
$j \in J$ are
pairwise different and linearly independent in $\cA_\LL$.
\end{Prop}

\begin{proof}
We treat $B_Q$ as a linear map $\Q^n \to \Q^n$.
For $\ba,\bb \in \Z^n$ define
$\ba \le \bb$ if there exists some $\be \in \Q_{\ge 0}^n$
such that 
$$
\ba = \bb + B_Q\be.
$$
We claim that this defines a partial order on $\Z^n$.
Clearly, $\le$ is reflexive and transitive.
Assume that $\ba \le \bb$ and $\bb \le \ba$.
Thus $\ba = \bb + B_Q\bff_1$ and
$\bb = \ba + B_Q\bff_2$ for some $\bff_1,\bff_2 \in \Q_{\ge 0}^n$.
It follows that $\ba = \ba + B_Q(\bff_1+\bff_2)$.
Thus $\bff_1+\bff_2 \in \Ker(B_Q)$.
Our assumption (i) yields that $\bff_1 = \bff_2 = 0$.
Thus $\ba = \bb$.
This shows that $\le$ is antisymmetric.

The partial order $\le$ on $\Z^n$ induces obviously
a partial order on the set of Laurent monomials
in the variables $x_1,\ldots,x_n$.
Namely, set $\bx^\ba \le \bx^\bb$ if
$\ba \le \bb$.
Let $\deg(\bx^\ba) := \ba$ be the \emph{degree} of
$\bx^\ba$.

Among the Laurent monomials $\bx^{g_\LL(\cM)+B_Q \be}$ occuring in the expression
$$
C_\LL(\cM) = \bx^{g_\LL(\cM)}\sum_{\be \in \N^n} \chi(\Gr_\be(M)) \bx^{B_Q \be} =
\sum_{\be \in \N^n} \chi(\Gr_\be(M)) \bx^{g_\LL(\cM)+B_Q \be}
$$
the monomial $\bx^{g_\LL(\cM)}$ is the unique monomial of
maximal degree.

For $\be = 0$ the Grassmannian $\Gr_\be(M)$ is just a point,
and $B_Q \be = 0$. 
Thus, if $\be = 0$, we have
$\chi(\Gr_\be(M))\bx^{B_Q \be} = 1$.
This shows that the Laurent monomial $\bx^{g_\LL(\cM)}$
really occurs as a non-trivial summand of $C_\LL(\cM)$.
In particular, we have $C_\LL(\cM) \not= C_\LL(\cN)$ if 
$g_\LL(\cM) \not= g_\LL(\cN)$.

Now let $\cM_1,\ldots,\cM_t$ be decorated representations
of $\LL$ with pairwise different $g$-vectors.
Assume that
$$
\lambda_1 C_\LL(\cM_1) + \cdots + \lambda_t C_\LL(\cM_t) = 0
$$
for some $\lambda_j \in \C$.
Without loss of generality we assume that $\lambda_j \not= 0$ for
all $j$.
There is a (not necessarily unique) index $s$ such that
$\bx^{g_\LL(\cM_s)}$ is maximal in the set
$\{ \bx^{g_\LL(\cM_j)} \mid 1 \le j \le t\}$.
It follow that the Laurent monomial $\bx^{g_\LL(\cM_s)}$
does not occur as a summand of any of the
Laurent polynomials $C_\LL(\cM_j)$ with
$j \not= s$.
(Here we use that the $g$-vectors of the decorated representations $\cM_j$ are pairwise different.)
This implies $\lambda_s = 0$, a contradiction.
Thus $C_\LL(\cM_1),\ldots,C_\LL(\cM_t)$ are linearly independent. 
\end{proof}

For $\ba = (a_1,\ldots,a_n)$ and $\bb = (b_1,\ldots,b_n)$
in $\Z^n$ set
$\ba \cdot \bb := a_1b_1 + \cdots + a_nb_n$.

Note that condition 
(d) in the following lemma coincides
with condition (i) in Proposition~\ref{li2}.

\begin{Lem}\label{liconditions}

For the conditions
\begin{itemize}

\item[(a)]
$\rk(B_Q) = n$.

\item[(b)]
Each row of $B_Q$ has at least
one non-zero entry, and
there are $n-\rk(B_Q)$ rows of
$B_Q$, which are non-negative linear combinations
of the remaining $\rk(B_Q)$ rows of $B_Q$.

\item[(c)]
$\Ima(B_Q) \cap \Q_{>0}^n \not= \varnothing$.

\item[(d)]
$\Ker(B_Q) \cap \Q_{\ge 0}^n  = 0$.

\end{itemize}
the implications
$$
{\rm (a)} \implies {\rm (b)} \implies {\rm (c)} \implies {\rm (d)}
$$
hold.
\end{Lem}

\begin{proof}
The implication ${\rm (a)} \implies {\rm (b)}$ is trivial.
Next, assume (b) holds.
Let $m := \rk(B_Q)$.
We denote the $j$th row of $B_Q$ by $r_j$.
By assumption there are pairwise different
indices $i_1,\ldots,i_m \in \{1,\ldots,n\}$ such that
for each $1 \le k \le n$ with
$k \notin \{ i_1,\ldots,i_m \}$ we have
$$
r_k = \lambda_1^{(k)} r_{i_1} + \cdots +
\lambda_m^{(k)} r_{i_m}
$$
for some non-negative rational numbers 
$\lambda_j^{(k)}$.
Since $r_k$ is non-zero, at least one of the $\lambda_j^{(k)}$
is positive.
Clearly, there is some $\be \in \Q^n$ such that
$r_{i_j} \cdot \be = 1$ for all $1 \le j \le m$.
(The $(k \times n)$-matrix with rows
$r_{i_1},\ldots,r_{i_m}$ has rank $m$.
Thus, we can see it as a surjective homomorphism
$\Q^n \to \Q^m$.)
Now observe that the $k$th entry of $B_Q\be$
is 
$\lambda_1^{(k)} + \cdots + \lambda_m^{(k)}$
for all $1 \le k \le n$ with $k \notin \{ i_1,\ldots,i_m \}$
and that this entry is positive.
It follows that $\Ima(B_Q) \cap \Q_{>0}^n \not= \varnothing$.

Finally, to show ${\rm (c)} \implies {\rm (d)}$ let
$\bb \in \Ima(B_Q) \cap \Q_{>0}^n$.
Thus there is some $\ba \in \Q^n$ such that $B_Q\ba = \bb$.
Since $B_Q$ is skew-symmetric, we get
$-\ba B_Q = \bb$.
Now let $\be \in \Ker(B_Q) \cap \Q_{\ge 0}^n$.
We get $B_Q\be = 0$, and therefore $-\ba B_Q\be = 
\bb \cdot \be = 0$.
Since $\bb$ has only positive entries and $\be$ has
only non-negative entries, we get $\be = 0$.
This finishes the proof.
\end{proof}

For the example, where $\LL$ is the path algebra of an affine quiver of type $\A_2$,
the main argument used in the proof of Proposition~\ref{li2}
can already be found in \cite[Section~6.1]{C}.
If we
replace condition (i) by condition (a), Proposition~\ref{li2} was first proved by Fu and Keller \cite[Corollary~4.4]{FK}. 
Essentially the same argument was later also used 
by Plamondon \cite{P1}.
That the Fu-Keller argument can be applied under condition
(b) was observed by Gei{\ss} and Labardini.
To any triangulation $T$ of a punctured Riemann surface with non-empty boundary,
one can associate a 2-acyclic quiver $Q_T$.
It is shown in \cite{GLaS} that there is always a triangulation $T$ such that the matrix $B_{Q_T}$ satisfies condition (b).

%%%%%%%%%%%%%%%%%%%%%%%%%%%%%%%%%%%%%%%%%%%%
%%%%%%%%%%%%%%%%%%%%%%%%%%%%%%%%%%%%%%%%%%%%

\section{Strongly reduced components of representation varieties}\label{sec5}

%%%%%%%%%%%%%%%%%%%%%%%%%%%%%%%%%%%%%%%%%%%%
%%%%%%%%%%%%%%%%%%%%%%%%%%%%%%%%%%%%%%%%%%%%

%%%%%%%%%%%%%%%%%%%%%%%%%%%%%%%%%%%%%%
\subsection{Strongly reduced components}
%%%%%%%%%%%%%%%%%%%%%%%%%%%%%%%%%%%%%%
Let $\LL = \CQ{Q}/I$ be a basic algebra, and let
$(\bd,\bv) \in \N^n \times \N^n$.
By $\irr_\bd(\LL)$ and $\decirr_{\bd,\bv}(\LL)$
we denote the set of irreducible components
of $\rep_\bd(\LL)$ and $\decrep_{\bd,\bv}(\LL)$, respectively.
For $Z \in \decirr_{\bd,\bv}(\LL)$ we write 
$\dimv(Z) := (\bd,\bv)$.
Let 
$$
\irr(\LL) = \bigcup_{\bd \in \N^n} \irr_\bd(\LL)
\text{\;\;\; and \;\;\;}
\decirr(\LL) = \bigcup_{(\bd,\bv) \in \N^n \times \N^n} \decirr_{\bd,\bv}(\LL).
$$
Note that any irreducible component $Z \in \decirr(\LL)$
can be seen
as an irreducible component in $\irr(\LL_\dec)$, where
$\LL_\dec := \LL \times \C \times \cdots \times \C$
is defined as the product of $\LL$ with $n$ copies 
of $\C$.
In fact, we can identify $\decrep(\LL)$ and $\rep(\LL_\dec)$.
Thus statements on 
varieties of representations can be carried over to varieties of decorated representations.

By definition we have
$$
\decrep_{\bd,\bv}(\LL) =
\{ (M,\C^\bv) \mid M \in \rep_\bd(\LL) \}.
$$
We have an isomorphism
$$
\decrep_{\bd,\bv}(\LL) \to \rep_\bd(\LL)
$$
of affine varieties mapping
$(M,\C^\bv)$ to $M$.
Thus the irreducible components of 
$\decrep_{\bd,\bv}(\LL)$ can be interpreted as irreducible components of
$\rep_\bd(\LL)$.
For $Z \in \decirr_{\bd,\bv}(\LL)$ let $\pi Z$ be the corresponding component in $\irr_\bd(\LL)$.

For $Z,Z_1,Z_2 \in \decirr(\LL)$ define
\begin{align*}
c_\LL(Z) &:= \min \{ \dim(Z) - \dim \cO(\cM) \mid \cM \in Z \},\\
e_\LL(Z) &:= \min \{ \dim \Ext_\LL^1(M,M) \mid \cM = (M,V) \in Z \},\\
\gEnd_\LL(Z) &:= \min \{ \dim \End_\LL(M) \mid \cM = (M,V) \in Z \},\\
\gHom_\LL(Z_1,Z_2) &:= \min \{ \dim \Hom_\LL(M_1,M_2) \mid \cM_i = (M_i,V_i) \in Z_i,\, i=1,2 \},\\
\gExt_\LL^1(Z_1,Z_2) &:= \min \{ \dim \Ext_\LL^1(M_1,M_2) \mid \cM_i = (M_i,V_i) \in Z_i,\, i=1,2 \}.
\end{align*}
For $Z,Z_1,Z_2 \in \decirr(\LL)$ there is a dense open subset $U$ of
$Z$ (resp. $Z_1 \times Z_2$) such that $E_\LL(\cM) = E_\LL(\cN)$ for all $\cM,\cN \in U$ (resp. $E_\LL(\cM_1,\cM_2) = E_\LL(\cN_1,\cN_2)$
for all $(\cM_1,\cM_2),(\cN_1,\cN_2) \in U$).
This follows from the upper semicontinuity of
the functions $\dim \Hom_\LL(-,?)$ and $\dim \Ext_\LL^1(-,?)$ proved
in \cite[Lemma~4.3]{CBS}.
For $\cM \in U$  (resp. $(\cM_1,\cM_2) \in U$) define 
$E_\LL(Z) := E_\LL(\cM)$ (resp. $E_\LL(Z_1,Z_2) := E_\LL(\cM_1,\cM_2)$).

Note that for $Z \in \decirr_{\bd,\bv}(\LL)$ we have
$$
c_\LL(Z) = \dim(Z) - \dim(G_\bd) + \gEnd_\LL(Z).
$$
This follows from Equation~(\ref{orbitdim}).

It is easy to construct examples of components $Z \in \decirr(\LL)$ such that $\gEnd_\LL(Z) \not= \gHom_\LL(Z,Z)$,
$e_\LL(Z) \not= \gExt_\LL^1(Z,Z)$ and $E_\LL(Z) \not= E_\LL(Z,Z)$.
Namely, let $\LL = \C Q$ be the path algebra of the Kronecker
quiver, and let $Z \in \decirr_{\bd,\bv}(\LL)$ with $\bd = (1,1)$ and
$\bv = (0,0)$. 
(Since $\LL$ is a path algebra of an acyclic quiver,
$\decrep_{\bd,\bv}(\LL)$ is irreducible for all $\bd,\bv$.)
An easy calculation shows that 
$\gEnd_\LL(Z) = e_\LL(Z) = E_\LL(Z) = 1$ and
$\gHom_\LL(Z,Z) = \gExt_\LL^1(Z,Z) = E_\LL(Z,Z) = 0$.
A further discussion of this example can be found in Section~\ref{sec9.4.3}.

The next lemma follows again from the upper semicontinuity
of $\dim \Hom_\LL(-,?)$ and $\dim \Ext_\LL^1(-,?)$.

\begin{Lem}\label{open}
For $Z,Z_1,Z_2 \in \decirr(\LL)$ the following hold:
\begin{itemize}

\item[(i)]
The sets
\begin{align*} 
& \{ \cM \in Z \mid \dim(Z) - \dim \cO(\cM)  = c_\LL(Z) \},\\
& \{ \cM = (M,V) \in Z \mid \dim \Ext_\LL^1(M,M) = e_\LL(Z) \},\\
& \{ \cM = (M,V) \in Z \mid \dim \End_\LL(M)  = \gEnd_\LL(Z) \}
\end{align*}
are open in $Z$.

\item[(ii)]
The sets
\begin{align*}
& \{ ((M_1,V_1),(M_2,V_2)) \in Z_1 \times Z_2 \mid
\dim \Hom_\LL(M_1,M_2)  = \gHom_\LL(Z_1,Z_2) \},\\
& \{ ((M_1,V_1),(M_2,V_2)) \in Z_1 \times Z_2 \mid
\dim \Ext_\LL^1(M_1,M_2)  = \gExt_\LL^1(Z_1,Z_2) \}
\end{align*}
are open in $Z_1 \times Z_2$.

\item[(iii)]
The sets
\begin{align*}
& \{ \cM \in Z \mid E_\LL(\cM)  = E_\LL(Z) \},\\
& \{ (\cM_1,\cM_2) \in Z_1 \times Z_2 \mid
E_\LL(\cM_1,\cM_2)  = E_\LL(Z_1,Z_2) \}
\end{align*}
are dense constructible subsets of $Z$ and $Z_1 \times Z_2$, 
respectively.

\end{itemize}
\end{Lem}

\begin{Lem}\label{lemmaineq} 
For $Z,Z_1,Z_2 \in \decirr(\LL)$ we have
$$
c_\LL(Z) \le e_\LL(Z) \le E_\LL(Z)
\text{\;\;\; and \;\;\;}
\gExt_\LL^1(Z_1,Z_2) \le E_\LL(Z_1,Z_2).
$$
\end{Lem}

\begin{proof}
Let $\bd = \dimv(\pi Z)$ and 
$\bd_i = \dimv(\pi Z_i)$ for $i=1,2$.
Choose some $p \ge 2|\bd|,|\bd_1|+|\bd_2|$.
By Lemma~\ref{trunclemma} we can regard all the representations
in $Z$, $Z_1$ and $Z_2$ as representations of $\LL_p$.
Thus we can interpret $Z$, $Z_1$ and $Z_2$ as irreducible components
in $\decirr(\LL_p)$.
Now Proposition~\ref{HomologicalE2} allows us to assume without loss of
generality that $\LL = \LL_p$.
Voigt's Lemma \cite[Proposition~1.1]{G} implies that
$c_\LL(Z) \le e_\LL(Z)$.
The Auslander-Reiten formula
$\Ext_\LL^1(M,N) \cong  {\rm D}\underline{\Hom}_\LL(\tau_\LL^-(N),M)$
yields 
$$
\dim \Ext_\LL^1(M,N) \le \dim \Hom_\LL(\tau_\LL^-(N),M).
$$
This implies $e_\LL(Z) \le E_\LL(Z)$ and
$\gExt_\LL^1(Z_1,Z_2) \le E_\LL(Z_1,Z_2)$.
(Here we used again Proposition~\ref{HomologicalE2}.)
\end{proof}

Following \cite{GLSChamber} we call
an irreducible component $Z \in \decirr(\LL)$ 
\emph{strongly reduced} provided
$$
c_\LL(Z) = e_\LL(Z) = E_\LL(Z).
$$
For example, if $\LL$ is finite-dimensional,
one can easily check that for any injective
$\LL$-module $I \in \rep(\LL)$ the closure of the
orbit $\cO(I,0)$ is a strongly reduced irreducible component. 
Similarly, it follows directly from the definitions 
that 
for all decorated representations
of the form $\cM = (0,V)$,
the closure of  $\cO(\cM)$ is a strongly reduced component. 
(In this case, $\cO(\cM)$ is just a
point, and it is equal to its closure.)

Let $\decirr^\sr_{\bd,\bv}(\LL)$ be the set of all strongly reduced components of $\decrep_{\bd,\bv}(\LL)$, and let
$$
\decirr^\sr(\LL) := \bigcup_{(\bd,\bv) \in \N^n \times \N^n}
\decirr^\sr_{\bd,\bv}(\LL).
$$

%%%%%%%%%%%%%%%%%%%%%%%%%%%%%%%%%%%%%%%%%%%%%%%%
\subsection{Decomposition theorems for irreducible components}
%%%%%%%%%%%%%%%%%%%%%%%%%%%%%%%%%%%%%%%%%%%%%%%%
An irreducible component $Z$ in $\irr(\LL)$ or $\decirr(\LL)$ is called
\emph{indecomposable} provided there exists a dense open
subset $U$ of $Z$, which contains only indecomposable representations or decorated representations,
respectively.
In particular, if $Z \in \decirr_{\bd,\bv}(\LL)$ is indecomposable,
then either $\bd = 0$ or $\bv = 0$.

Given irreducible components $Z_i$ of $\decrep_{\bd_i,\bv_i}(\LL)$
for $1 \le i \le t$, let $(\bd,\bv) := (\bd_1,\bv_1) + \cdots + (\bd_t,\bv_t)$ and let 
$$
Z_1 \oplus \cdots \oplus Z_t
$$
be the points of $\decrep_{\bd,\bv}(\LL)$, which are isomorphic to
$M_1 \oplus \cdots \oplus M_t$ with $M_i \in Z_i$ for
$1 \le i \le t$.
The Zariski closure of $Z_1 \oplus \cdots \oplus Z_t$ in
$\decrep_{\bd,\bv}(\LL)$ is
denoted by 
$$
\overline{Z_1 \oplus \cdots \oplus Z_t}.
$$
It is quite easy to show that $\overline{Z_1 \oplus \cdots \oplus Z_t}$ is an irreducible closed subset of $\decrep_{\bd,\bv}(\LL)$, but
in general it is not an irreducible component.

\begin{Thm}[{\cite{CBS}}]\label{decomp1}
For $Z_1,\ldots,Z_t \in \decirr(\LL)$ the following are equivalent:
\begin{itemize}

\item[(i)]
$\overline{Z_1 \oplus \cdots \oplus Z_t}$ is an irreducible 
component.

\item[(ii)]
$\gExt_\LL^1(Z_i,Z_j) = 0$ for all $i \not= j$.

\end{itemize}
Furthermore, the following hold:
\begin{itemize}

\item[(iii)]
Each irreducible component $Z \in \decirr(\LL)$
can be written as 
$Z = \overline{Z_1 \oplus \cdots \oplus Z_t}$ with
$Z_1,\ldots,Z_t$ indecomposable irreducible components
in $\decirr(\LL)$.
Suppose that
$$
\overline{Z_1 \oplus \cdots \oplus Z_t} = \overline{Z_1' \oplus \cdots \oplus Z_s'}
$$ 
is an irreducible component with $Z_i$ and $Z_i'$ indecomposable irreducible components in $\decirr(\LL)$ 
for
all $i$.
Then $s=t$ and there is a bijection $\sigma\colon \{ 1,\ldots,t \} \to
\{ 1,\ldots,s \}$ such that $Z_i = Z_{\sigma(i)}'$ for all $i$.

\end{itemize}
\end{Thm}

The next lemma is an easy exercise.

\begin{Lem}\label{negsimple}
For $1 \le i \le n$ and any decorated representation
$\cM = (M,V)$ of $\LL$ we have
$$
E_\LL(\cM,\cS_i^-) = \dim(M_i)
\text{\;\;\; and \;\;\;}
E_\LL(\cS_i^-,\cM) = 0.
$$
\end{Lem}

\begin{Cor}\label{cor5.5}
If
$\decirr_{\bd,\bv}^\sr(\LL) \not= \varnothing$, then we
have $d_iv_i = 0$ for all $1 \le i \le n$.
\end{Cor}

\begin{proof}
Let $Z \in \decirr_{\bd,\bv}^\sr(\LL)$ for some $\bd,\bv$,
and let $Z'$ be the corresponding irreducible component
of $\decirr_{\bd,0}(\LL)$.
We clearly have $c_\LL(Z) = c_\LL(Z')$ and 
$E_\LL(Z) = E_\LL(Z') + d_1v_1 + \cdots + d_nv_n$.
Using Lemma~\ref{lemmaineq} we obtain
$c_\LL(Z) = c_\LL(Z') \le E_\LL(Z') \le E_\LL(Z)$.
Since $c_\LL(Z) = E_\LL(Z)$, the result follows.
\end{proof}

\begin{Lem}\label{truncsr}
Let $Z \in \decirr_{\bd,\bv}(\LL)$, and assume that
$p > |\bd|$.
Then the following are equivalent:
\begin{itemize}

\item[(i)]
$Z \in \decirr^\sr(\LL)$.

\item[(ii)]
$Z \in \decirr^\sr(\LL_p)$.

\end{itemize}
\end{Lem}

\begin{proof}
Since $p > |\bd|$, we can apply 
Lemma~\ref{trunclemma} and Proposition~\ref{HomologicalE2}
and get $c_{\LL_p}(Z) = c_\LL(Z)$ and
$E_{\LL_p}(Z) = E_\LL(Z)$.
This yields the result.
\end{proof}

The additivity of the functor $\Hom_\LL(-,?)$ 
and upper semicontinuity imply the following lemma.

\begin{Lem}\label{dimlemma1}
Let $Z,Z_1,Z_2 \in \decirr(\LL)$.
Suppose that $Z = \overline{Z_1 \oplus Z_2}$.
Then the following hold:
\begin{itemize}

\item[(i)]
$\gEnd_\LL(Z) = \gEnd_\LL(Z_1) + \gEnd_\LL(Z_2) + \gHom_\LL(Z_1,Z_2) + \gHom_\LL(Z_2,Z_1)$.

\item[(ii)]
$E_\LL(Z) = E_\LL(Z_1) + E_\LL(Z_2) + E_\LL(Z_1,Z_2) + 
E_\LL(Z_2,Z_1)$.

\end{itemize}
\end{Lem}

Recall that
for $\ba = (a_1,\ldots,a_n)$ and $\bb = (b_1,\ldots,b_n)$ in
$\Z^n$ we defined
$\ba \cdot \bb := a_1b_1 + \cdots + a_nb_n$.
The following lemma is obvious.

\begin{Lem}\label{dimlemma2}
Let $\bd,\bd_1,\bd_2 \in \N^n$ with $\bd = \bd_1 + \bd_2$.
Then 
$$
\dim(G_\bd) - \dim(G_{\bd_1}) - \dim(G_{\bd_2}) 
= 2 (\bd_1 \cdot \bd_2).
$$
\end{Lem}

\begin{Lem}\label{dimlemma3}
Let
$Z,Z_1,Z_2 \in \decirr(\LL)$ with
$Z = \overline{Z_1 \oplus Z_2}$.
We have
$$
\dim(Z) =
\dim(Z_1) + \dim(Z_2) 
+ 2(\dimv(\pi Z_1) \cdot \dimv(\pi Z_2)) 
- \gHom_\LL(Z_1,Z_2) - \gHom_\LL(Z_2,Z_1).
$$
\end{Lem}

\begin{proof}
For $i=1,2$ let
$(\bd_i,\bv_i) := \dimv(Z_i)$, and let
$(\bd,\bv) := \dimv(Z)$.
We have $\dimv(Z) = \dimv(Z_1) + \dimv(Z_2)$ and
$\dimv(\pi Z_i) = \bd_i$.
The map
$$
f\colon G_\bd \times Z_1 \times Z_2
\to Z
$$
defined by
$$
(g,(M_1,\C^{\bv_1}),(M_2,\C^{\bv_2})) \mapsto
(g.(M_1 \oplus M_2),\C^\bv)
$$
is a morphism of affine varieties.
For $(\cM_1,\cM_2) \in Z_1 \times Z_2$ define
$$
f_{\cM_1,\cM_2}\colon G_\bd \times \cO(\cM_1) \times \cO(\cM_2)
\to \cO(\cM_1 \oplus \cM_2)
$$
by $(g,\cN_1,\cN_2) \mapsto f(g,\cN_1,\cN_2)$.
The fibres of $f_{\cM_1,\cM_2}$ are
of dimension 
$$
d_{\cM_1,\cM_2} := \dim(G_\bd) + \dim \cO(\cM_1) + \dim \cO(\cM_2)
- \dim \cO(\cM_1 \oplus \cM_2).
$$
Using Equation~(\ref{orbitdim}), 
an easy calculation yields
$$
d_{\cM_1,\cM_2} = \dim(G_{\bd_1}) + \dim(G_{\bd_2}) +
\dim \Hom_\LL(M_1,M_2) + \dim \Hom_\LL(M_2,M_1).
$$
Let $\cM$ be in the image of 
$f$.
We want to compute the dimension of the fibre $f^{-1}(\cM)$.
Let 
$$
\cT :=  \{ 
\cO(\cN_1) \times \cO(\cN_2) \subseteq 
Z_1 \times Z_2 \mid  \cN_1 \oplus \cN_2 \cong \cM \}.
$$
It follows from
the Krull-Remak-Schmidt Theorem that $\cT$ is a finite set.
Thus the fibre $f^{-1}(\cM)$ is the disjoint
union of the fibres $f_{\cN_1,\cN_2}^{-1}(\cM)$, where
$\cO(\cN_1) \times \cO(\cN_2)$ runs through $\cT$.
So we get
$$
\dim(f^{-1}(\cM)) = \max\{ d_{\cN_1,\cN_2} \mid
\cO(\cN_1) \times \cO(\cN_2) \in \cT \}.
$$
Thus
by upper semicontinuity
there is a dense open subset $V \subseteq Z$ such that all fibres 
$f^{-1}(\cM)$ with $\cM \in V$ have 
dimension
$$
d_{Z_1,Z_2} :=
\dim(G_{\bd_1}) + \dim(G_{\bd_2}) +
\gHom_\LL(Z_1,Z_2) + \gHom_\LL(Z_2,Z_1).
$$
We have
$$
\dim(Z) + d_{Z_1,Z_2} = \dim(G_\bd) + \dim(Z_1) + \dim(Z_2).
$$
Using Lemma~\ref{dimlemma2} we get
$$
\dim(Z) = \dim(Z_1) + \dim(Z_2) + 2(\bd_1 \cdot \bd_2)
-
\gHom_\LL(Z_1,Z_2) - \gHom_\LL(Z_2,Z_1).
$$
This finishes the proof.
\end{proof}

\begin{Lem}\label{cadditive}
For
$Z,Z_1,Z_2 \in \decirr(\LL)$ with
$Z = \overline{Z_1 \oplus Z_2}$
we have
$$
c_\LL(Z) = c_\LL(Z_1) + c_\LL(Z_2).
$$
\end{Lem}

\begin{proof}
For $i=1,2$ let
$(\bd_i,\bv_i) := \dimv(Z_i)$, and let
$(\bd,\bv) := \dimv(Z)$.
We get
\begin{align*}
c_\LL(Z) &= \dim(Z) - \dim(G_\bd) + \gEnd_\LL(Z)
\\
&=\dim(Z_1) + \dim(Z_2) - 
\dim(G_{\bd_1}) - \dim(G_{\bd_2}) + \gEnd_\LL(Z_1) + \gEnd_\LL(Z_2)\\
&= c_\LL(Z_1) + c_\LL(Z_2). 
\end{align*}
The first equality follows directly from the definition of $c_\LL(Z)$.
The second equality uses Lemma~\ref{dimlemma1}(i) and Lemma~\ref{dimlemma3}.
\end{proof}

The following result is a version of Theorem~\ref{decomp1}
for strongly reduced components.

\begin{Thm}\label{decomp2}
For $Z_1,\ldots,Z_t \in \decirr(\LL)$ the following are equivalent:
\begin{itemize}

\item[(i)]
$\overline{Z_1 \oplus \cdots \oplus Z_t}$ is a strongly reduced irreducible component.

\item[(ii)]
Each $Z_i$ is strongly reduced and
$E_\LL(Z_i,Z_j) = 0$ for all $i \not= j$.

\end{itemize}
\end{Thm}

\begin{proof}
Without loss of generality assume that $t =2$.
The general case follows by induction.
Let $Z_1 \in \decirr_{\bd_1,\bv_1}(\LL)$ and 
$Z_2 \in \decirr_{\bd_2,\bv_2}(\LL)$.

Assume that $Z := \overline{Z_1 \oplus Z_2}$ is a strongly
reduced component.
Thus we have $c_\LL(Z) = E_\LL(Z)$.
Applying Lemma~\ref{cadditive} and
Lemma~\ref{dimlemma1}(ii) this implies
$$
c_\LL(Z_1)+c_\LL(Z_2) =
E_\LL(Z_1) + E_\LL(Z_2) + E_\LL(Z_1,Z_2) + E_\LL(Z_2,Z_1).
$$
Since $c_\LL(Z_i) \le E_\LL(Z_i)$ we get $E_\LL(Z_1,Z_2) = E_\LL(Z_2,Z_1) = 0$
and $c_\LL(Z_i) = E_\LL(Z_i)$.
Thus (i) implies (ii).

To show the converse,
assume that $Z_1$ and $Z_2$ are strongly reduced with
$E_\LL(Z_1,Z_2) = E_\LL(Z_2,Z_1) = 0$.
We claim that
$$
c_\LL(Z) = c_\LL(Z_1) + c_\LL(Z_2) = E_\LL(Z_1) + E_\LL(Z_2) = E_\LL(Z).
$$
For the first equality we use Lemma~\ref{cadditive},
the second equality is just our assumption that $Z_1$ and $Z_2$ are strongly reduced.
Finally, the third equality follows from Lemma~\ref{dimlemma1} together with our assumption
that $E_\LL(Z_1,Z_2)$ and $E_\LL(Z_2,Z_1)$ are both zero.
Thus $Z$ is strongly reduced.
\end{proof}

Note that Theorems~\ref{decomp1} and \ref{decomp2}
imply that each  $Z \in \decirr^\sr(\LL)$
is of the form $Z = \overline{Z_1 \oplus \cdots \oplus Z_t}$
with $Z_i \in \decirr^\sr(\LL)$ and $Z_i$ indecomposable for all $i$.

%%%%%%%%%%%%%%%%%%%%%%%%%
\subsection{Generic $g$-vectors}
%%%%%%%%%%%%%%%%%%%%%%%%%
For $Z \in \decirr(\LL)$ 
there is a dense open subset $U$ of $Z$ such that
$g_\LL(\cM) = g_\LL(\cN)$ for all $\cM,\cN \in U$.
This follows again by upper semicontinuity.
For $\cM \in U$ let
$$
g_\LL(Z) := g_\LL(\cM)
$$ 
be the \emph{generic $g$-vector} of $Z$.

The next lemma follows directly 
from upper semicontinuity and
Lemma~\ref{additive}(i).

\begin{Lem}\label{gadditive}
For
$Z,Z_1,Z_2 \in \decirr(\LL)$ with
$Z = \overline{Z_1 \oplus Z_2}$
we have
$$
g_\LL(Z) = g_\LL(Z_1) + g_\LL(Z_2).
$$
\end{Lem}

\begin{Lem}
For $Z \in \decirr_{\bd,\bv}^\sr(\LL)$ we have
$$
\bd \cdot g_\LL(Z) = \dim(Z) - \dim(G_\bd).
$$
\end{Lem}

\begin{proof}
It follows from the definitions that
$$
E_\LL(Z) = \gEnd_\LL(Z) + \bd \cdot g_\LL(Z),
$$ 
and we have
$$
c_\LL(Z) = 
\dim(Z) - \dim(G_\bd) + \gEnd_\LL(Z).
$$
Now the claim follows, since $c_\LL(Z) = E_\LL(Z)$.
\end{proof}

\begin{Cor}
Let $Z \in \decirr_{\bd,\bv}^\sr(\LL)$ with $\bd \not= 0$.
If $E_\LL(Z) = 0$, then
$$
\bd \cdot g_\LL(Z) = - \gEnd_\LL(Z) < 0.
$$
\end{Cor}

%%%%%%%%%%%%%%%%%%%%%%%%%%%%%%%%%%%%%%%%
\subsection{Parametrization of strongly
reduced components}
%%%%%%%%%%%%%%%%%%%%%%%%%%%%%%%%%%%%%%%%
Let $\LL = \CQ{Q}/I$ be a finite-dimensional basic algebra.
Plamondon \cite{P2} constructs a map
$$
P_\LL\colon \decirr(\LL) \to \Z^n,
$$
which can be defined as follows:
Let $Z \in \decirr(\LL)$.
Then there exist injective $\LL$-modules
$I_0^\LL(Z)$ and $I_1^\LL(Z)$, which are uniquely determined
up to isomorphism, and a dense open subset $U \subseteq \pi Z$
such that for each representation $M \in U$ we have
$I_0^\LL(M) = I_0^\LL(Z)$ and $I_1^\LL(M) = I_1^\LL(Z)$.
For $Z \in \decirr_{\bd,\bv}(\LL)$ define
$$
P_\LL(Z) := 
-\dimv(\soc(I_0^\LL(Z))) + \dimv(\soc(I_1^\LL(Z)))
+ \bv.
$$
Let 
$$
P_\LL^\sr\df \decirr^\sr(\LL) \to \Z^n
$$
be the restriction of $P_\LL$ to $\decirr^\sr(\LL)$.

For a representation $M$ let $\add(M)$ be the category
of all finite direct sums of direct summands of $M$.
Plamondon \cite{P2} 
obtains the following striking result.

\begin{Thm}[Plamondon]\label{plamondon1}
For any finite-dimensional basic algebra $\LL$
the following hold:
\begin{itemize}

\item[(i)]
$$
P_\LL^\sr\df \decirr^\sr(\LL) \to \Z^n
$$ 
is bijective.

\item[(ii)]
For every 
$Z \in \decirr^\sr(\LL)$ 
we have
$$
\add(I_0^\LL(Z)) \cap \add(I_1^\LL(Z)) = 0.
$$ 

\end{itemize}
\end{Thm}

Note that Plamondon works with irreducible components,
and not with decorated irreducible components.
But his results translate easily from one concept to the other.

We now generalize Theorem~\ref{plamondon1}(i) to arbitrary
basic algebras $\LL$.
It turns out that $\decirr^\sr(\LL)$ is in general no longer
parametrized by $\Z^n$ but by a subset of $\Z^n$.
Our proof is based on Plamondon's result and uses 
additionally truncations of basic algebras.

For a basic algebra $\LL$
let 
$$
G_\LL\df \decirr(\LL) \to \Z^n
$$
be the map, which sends $Z \in \decirr(\LL)$
to the generic $g$-vector $g_\LL(Z)$ of $Z$.

For finite-dimensional $\LL$, it follows immediately from Lemma~\ref{HomologicalE1}
that $G_\LL = P_\LL$.
Let 
$$
G_\LL^\sr\df \decirr^\sr(\LL) \to \Z^n
$$
be the restriction of $G_\LL$ to $\decirr^\sr(\LL)$.

For a basic algebra $\LL$
let 
$$
\decirr_{<p}(\LL)
$$  
be the set of irreducible components $Z \in \decirr(\LL)$ such that $(\bd,\bv) := \dimv(Z)$ satisfies
$|\bd| < p$.  
Define 
$$
\decirr_{<p}^\sr(\LL) := \decirr_{<p}(\LL) \cap 
\decirr^\sr(\LL).
$$

\begin{Lem}\label{reduction}
For a basic algebra $\LL = \CQ{Q}/I$ the following hold:
\begin{itemize}

\item[(i)]
For all $p \le q$ we have
$$
\decirr_{<p}^\sr(\LL_p) \subseteq  \decirr_{<q}^\sr(\LL_q)
\subseteq  \decirr^\sr(\LL).
$$

\item[(ii)]
We have
$$
\decirr^\sr(\LL) =
\bigcup_{p > 0} \decirr_{<p}^\sr(\LL_p).
$$

\item[(iii)]
We have 
$$
\decirr^\sr(\overline{\LL}) = \decirr^\sr(\LL).
$$

\end{itemize}
\end{Lem}

\begin{proof}
Let $Z \in \decirr_{\bd,\bv}(\LL)$, and let
$p > |\bd|$.
Thus we have $Z \in \decirr_{\bd,\bv}(\LL_p)$ and 
$Z \in \decirr_{<p}(\LL_p)$.
Furthermore, we have $c_{\LL_p}(Z) = c_\LL(Z)$ and
$E_{\LL_p}(Z) = E_\LL(Z)$.
Thus $Z \in \decirr^\sr(\LL)$ if and only if $Z \in \decirr^\sr(\LL_p)$.
This yields (i) and (ii).
Recall that 
$$
\overline{I} = \bigcap_{p \ge 0} (I+\m^p)
$$
and $\LL_p = \CQ{Q}/(I+\m^p)$.
For $p \ge 2$ it is easy to check that
$$
\overline{I} + \m^p = I + \m^p.
$$
This implies $\overline{\LL}_p = \LL_p$.
Now (ii) implies (iii).
\end{proof}

\begin{Thm}\label{parametrization}
For a basic algebra $\LL$ the following hold:
\begin{itemize}

\item[(i)]
The map
$$
G_\LL^\sr\df \decirr^\sr(\LL) \to \Z^n
$$
is injective.

\item[(ii)]
The following are equivalent:
\begin{itemize}

\item[(a)]
$G_\LL^\sr$ is surjective.

\item[(b)]
$\overline{\LL}$ is finite-dimensional.

\end{itemize}

\end{itemize}
\end{Thm}

\begin{proof}
Since $\LL_p$ is finite-dimensional for all $p$, we
know from Plamondon's Theorem~\ref{plamondon1}(i) that 
$$
G_{\LL_p}^\sr\colon \decirr^\sr(\LL_p) \to
\Z^n
$$ 
is bijective. 
Now Lemma~\ref{reduction} yields that
the map
$$
G_\LL^\sr\colon \decirr^\sr(\LL) \to \Z^n
$$
sends
$Z \in \decirr_{<p}^\sr(\LL_p)$ to
$G_{\LL_p}^\sr(Z)$, and that
$G_\LL^\sr$ is injective.
This proves (i).
Theorem \ref{plamondon1}(i) together with Lemma~\ref{reduction}(iii)
says that (b) implies (a).
(Recall that $\decrep(\LL) = \decrep(\overline{\LL})$.
This implies that $G_\LL = G_{\overline{\LL}}$ and
$G_\LL^\sr = G_{\overline{\LL}}^\sr$.
Note also
that for every $Z \in \decirr(\LL)$ we have
$G_\LL(Z) = G_{\LL_p}(Z)$ for some large enough $p$.)

To show the converse, assume that $\overline{\LL}$ is infinite dimensional.
Recall that
$$
\overline{\LL} =
\underleftarrow{\lim}(\LL_p)
$$
and
$$
\overline{I}_i = \bigcup_{p \ge 2} I_{i,p}
$$
where 
$\overline{I}_i$ is the injective envelope of the simple $\overline{\LL}$-module $S_i$, and
$I_{i,p}$ is the injective envelope of $S_i$ considered as
a
$\LL_p$-module, and $\LL_p = \CQ{Q}/(I+\m^p)$.
We have $I_{i,p} = D(e_i\LL_p)$,
where $D = \Hom_K(-,K)$ is the $K$-dual functor.
This implies
$$
\dim(\LL_p) = \sum_{i=1}^n \dim(I_{i,p}).
$$
It follows that there exists
some $1 \le i \le n$ such that $\overline{I}_i$
is infinite dimensional.

As a vector space, $e_i\LL_p$ is generated by the
residue classes $a + (I+\m^p)$ of all paths $a$ in $Q$ with
$t(a) = i$.
We have 
$$
I_{i,2} = D(e_i\LL_2) = D(e_i(\CQ{Q}/(I+\m^2)))
= D(e_i(\CQ{Q}/\m^2)).
$$
This implies $\dim(I_{i,2}) \ge 2$. (Otherwise, there is no
arrow $a$ with $t(a) = i$, which implies $I_{i,p} = I_{i,2}$ for
all $p \ge 2$, a contradiction since $\overline{I}_i$ is infinite
dimensional.)
Now suppose that $I_{i,p-1} = I_{i,p}$ for some $p \ge 3$.
This implies $e_i\LL_{p-1} = e_i\LL_p$.
Thus we have $e_i(I+\m^{p-1}) = e_i(I+\m^p)$.
It follows that
$$
e_i(I+\m^{p+1}) = e_iI+e_i(I+\m^p)\m
= e_iI+e_i(I+\m^{p-1})\m = e_i(I+\m^p). 
$$
In other words, we have $I_{i,p+1} = I_{i,p}$.
By induction we get $I_{i,q} = I_{i,p-1}$ for all
$q \ge p$, a contradiction since $\overline{I}_i$ is infinite
dimensional.
Thus we proved that
$\dim(I_{i,p}) \ge p$ for all $p \ge 2$.

Now assume that $-\be_i$ is in the image of $G_\LL^\sr$.
(Here $\be_i$ denotes the $i$th standard basis vector of
$\Z^n$.)
In other words,
there is some $Z \in \decirr^\sr(\LL)$ such that
$G_\LL^\sr(Z) = -\be_i$.
By Lemma~\ref{reduction}(ii) we know that
$Z \in \decirr_{<p}^\sr(\LL_p)$ for some $p \ge 1$.
Since $g_\LL(Z) = -\be_i$, we have $I_0^{\LL_p}(Z) = I_{i,p}$
and $I_1^{\LL_p}(Z) = 0$.
(Here we use Theorem~\ref{plamondon1}(ii).)
This implies that $Z$ is the closure of the orbit of the decorated
representation $(I_{i,p},0)$.
But $\dim(I_{i,p}) \ge p$ and the dimension
of all representations in $Z$ is strictly smaller than $p$, 
a contradiction.
\end{proof}

%%%%%%%%%%%%%%%%%%%%%%%%%%%%%%%%%%%%%%%%%%%%
%%%%%%%%%%%%%%%%%%%%%%%%%%%%%%%%%%%%%%%%%%%%

\section{Component graphs and $CC$-clusters}\label{sec6}

%%%%%%%%%%%%%%%%%%%%%%%%%%%%%%%%%%%%%%%%%%%%
%%%%%%%%%%%%%%%%%%%%%%%%%%%%%%%%%%%%%%%%%%%%

%%%%%%%%%%%%%%%%%%%%%%%%%%%%%%%%%%%%%
\subsection{The graph of strongly reduced components}\label{compgraph}
%%%%%%%%%%%%%%%%%%%%%%%%%%%%%%%%%%%%%
Let $\LL$ be a basic algebra.
In \cite{CBS} the \emph{component graph} $\Gamma(\irr(\LL))$ of $\LL$ is defined as follows:
The vertices of $\Gamma(\irr(\LL))$ are the indecomposable irreducible
components in $\irr(\LL)$.
There is an edge between (possibly equal) vertices $Z_1$ and $Z_2$ if $\gExt_\LL^1(Z_1,Z_2) = \gExt_\LL^1(Z_2,Z_1) = 0$.

We want to define an analogue of $\Gamma(\irr(\LL))$ for 
strongly reduced components.
The graph 
$\Gamma(\decirr^\sr(\LL))$
of strongly reduced components 
has as vertices the indecomposable components
in $\decirr^\sr(\LL)$,
and there is an edge between (possibly equal) vertices
$Z_1$ and $Z_2$ if
$E_\LL(Z_1,Z_2) = E_\LL(Z_2,Z_1) = 0$.

%%%%%%%%%%%%%%%%%%%%%%%%%
\subsection{Component clusters}
%%%%%%%%%%%%%%%%%%%%%%%%%
Let $\Gamma$ be a graph, and let  $\Gamma_0$ be the set of vertices of $\Gamma$.
We allow only single edges, and we allow loops, i.e. edges
from a vertex to itself.
For a subset $\cU \subseteq \Gamma_0$ let $\Gamma_\cU$ be the full subgraph,
whose set of vertices is $\cU$. 
The subgraph $\Gamma_\cU$ is
\emph{complete} if for each $i,j \in \cU$ with $i \not= j$
there is an edge between $i$ and $j$.
A complete subgraph $\Gamma_\cU$ is \emph{maximal} if
for each complete subgraph $\Gamma_{\cU'}$ with $\cU \subseteq \cU'$
we have $\cU = \cU'$.
We call a subgraph $\Gamma_\cU$ \emph{loop-complete} if $\Gamma_\cU$ is complete and
there is a  loop at each vertex in $U$.

The set of vertices of a 
maximal complete subgraph of $\Gamma := \Gamma(\decirr^\sr(\LL))$
is called a
\emph{component cluster} of $\LL$.
A component cluster $\cU$ of $\LL$ is 
$E$-\emph{rigid} provided $E_\LL(Z)=0$ 
for all $Z \in \cU$.
(Recall that there is a loop at a vertex $Z$ of
$\Gamma$ if and only if $E_\LL(Z,Z) = 0$.
If $E_\LL(Z) = 0$, then $E_\LL(Z,Z) = 0$, 
but the converse does not hold.)
Clearly, each $E$-rigid component cluster is loop-complete.

\begin{Thm}\label{upperbound}
For each loop-complete subgraph $\Gamma_\cU$ of 
$\Gamma := \Gamma(\decirr^\sr(\LL))$
we have $|\cU| \le n$.
\end{Thm}

\begin{proof}
Assume that
$Z_1,\ldots,Z_{n+1}$ 
are pairwise different vertices of a loop-complete
subgraph $\Gamma_J$ of $\Gamma(\decirr^\sr(\LL))$.
For $1 \le i \le n+1$ let $g_\LL(Z_i)$ 
be the generic $g$-vector of $Z_i$.
Since $\Gamma_J$ is loop-complete we know by Theorem~\ref{decomp2}
that 
$$
Z_\ba := \overline{Z_1^{a_1} \oplus \cdots \oplus Z_{n+1}^{a_{n+1}}}
$$
is again a strongly reduced component for each
$\ba = (a_1,\ldots,a_{n+1}) \in \N^{n+1}$.
By the additivity of $g$-vectors we get
$$
g_\LL(Z_\ba) = a_1 g_\LL(Z_1) + \cdots + a_{n+1}g_\LL(Z_{n+1}).
$$
Furthermore, we know from Theorem~\ref{decomp1} that $Z_\ba = Z_\bb$
if and only if $\ba = \bb$. 
Now one can essentially copy the proof of 
\cite[Theorem~1.1]{GS} to show that there are
$\ba,\bb \in \N^{n+1}$ with
$g_\LL(Z_\ba) = g_\LL(Z_\bb)$ but $\ba \not= \bb$.
By Theorem~\ref{parametrization} different strongly
reduced components have different $g$-vectors.
Thus we have a contradiction.
\end{proof}

\begin{Cor}
Let $\LL$ be a finite-dimensional basic
algebra.
Let $M$ be a representation of $\LL$ with
$\Hom_\LL(\tau_\LL^-(M),M) = 0$.
Then $M$ has at most $n$ isomorphism classes
of indecomposable direct summands.
\end{Cor}

The following conjecture might be a bit too optimistic.
But 
it is true for $\LL = \CQ{Q}$ the path algebra of an acyclic quiver $Q$, see \cite[Corollary~21]{DW} 
and Section~\ref{srhereditary}.

\begin{Conj}\label{conjgraph1}
For any basic algebra $\LL$ the following hold:
\begin{itemize}

\item[(i)]
The component clusters of  $\LL$
have cardinality at most $n$.

\item[(ii)]
The $E$-rigid component clusters of $\LL$ are exactly the
component clusters of cardinality $n$.

\end{itemize}
\end{Conj}

%%%%%%%%%%%%%%%%%%%%%%%%%%%%%%%
\subsection{$E$-rigid representations}
%%%%%%%%%%%%%%%%%%%%%%%%%%%%%%%
After most of this work was done, we learned that
Iyama and Reiten \cite{IR} obtained some beautiful results on
socalled $\tau$-\emph{rigid modules} over finite-dimensional algebras, which fit perfectly
into the framework of Caldero-Chapoton algebras.

Adapting their terminology to decorated representations of basic algebras,
a decorated representation $\cM$ of a basic algebra $\LL$
is called $E$-\emph{rigid} provided $E_\LL(\cM) = 0$.
The following theorem is just a reformulation of Iyama
and Reiten's results on $\tau$-rigid modules.
Part (i) follows also directly from the
more general statement in Theorem~\ref{upperbound}.

For $\cM \in \decrep(\LL)$ let 
$\Sigma(\cM)$ be the number of isomorphism classes
of
indecomposable direct summands of $\cM$.

\begin{Thm}[{\cite{IR}}]\label{IyamaReiten}
Let $\LL = \CQ{Q}/I$ be a finite-dimensional basic algebra.
For $\cM \in \decrep(\LL)$ the following hold:
\begin{itemize}

\item[(i)]
If $\cM$ is $E$-rigid, then $\Sigma(\cM) \le n$.

\item[(ii)]
For each $E$-rigid $\cM \in \decrep(\LL)$ there exists
some $\cN \in \decrep(\LL)$ such that
$\cM \oplus \cN$ is $E$-rigid and 
$\Sigma(\cM \oplus \cN) = n$.

\item[(iii)]
For each $E$-rigid $\cM \in \decrep(\LL)$ with $\Sigma(\cM) = n-1$ there are exactly two non-isomorphic
indecomposable decorated representations
$\cN_1,\cN_2 \in \decrep(\LL)$ such that $\cM \oplus \cN_i$
is $E$-rigid and $\Sigma(\cM \oplus \cN_i) = n$ for
$i=1,2$.
 
\end{itemize}
\end{Thm}

It is easy to find examples of infinite dimensional
basic algebras $\LL$ such that 
Theorem~\ref{IyamaReiten}(iii) does not hold, see Section~\ref{9.3.1}.

A basic algebra $\LL$ is \emph{representation-finite}
if there are only finitely many isomorphism classes
of indecomposable representations in $\rep(\LL)$.
One easily checks that $\LL$ is finite-dimensional
in this case.

\begin{Cor}\label{cor6.5}
Assume that $\LL$ is a representation-finite basic algebra.
Then the following hold:
\begin{itemize}

\item[(i)]
Each component cluster of $\LL$ is $E$-rigid.

\item[(ii)]
Each component cluster of $\LL$ has cardinality $n$.

\item[(iii)]
There is bijection between the set of isomorphism classes
of $E$-rigid representation of $\LL$ to
the set $\decirr^\sr(\LL)$ of strongly reduced components.
Namely, one maps an $E$-rigid representation $\cM$
to the closure of the orbit $\cO(\cM)$.

\end{itemize}
\end{Cor}

\begin{proof}
Since $\LL$ is representation-finite, every irreducible  component $Z \in \decirr(\LL)$ is a union of finitely many orbits, and exactly one of these orbits has do be dense in 
$Z$. 
Thus we have $c_\LL(Z) = 0$. 
This implies (i) and (iii).
Now (ii) follows directly from Theorem~\ref{IyamaReiten}(ii).
\end{proof}

%%%%%%%%%%%%%%%%%%%%%%%%%%%%%%%%%%%%%
\subsection{Generic Caldero-Chapoton functions}\label{defgenericCC}
%%%%%%%%%%%%%%%%%%%%%%%%%%%%%%%%%%%%%
For each $(\bd,\bv) \in \N^n \times \N^n$ let
$$
C_{\bd,\bv}\colon \decrep_{\bd,\bv}(\LL) \to \Z[x_1^\pm,\ldots,x_n^\pm]
$$
be the function defined by $\cM \mapsto C_\LL(\cM)$.
The map $C_{\bd,\bv}$ is a constructible function.
In particular, the image of $C_{\bd,\bv}$ is finite.
Thus 
for an irreducible component $Z \in \decirr_{\bd,\bv}(\LL)$
there exists a dense open subset $U \subseteq Z$ such
that $C_{\bd,\bv}$ is constant on $U$.
Define 
$$
C_\LL(Z) := C_\LL(\cM)
$$ 
with $\cM$ any representation in $U$.
The element $C_\LL(Z)$ depends only on $Z$ and not on the
choice of $U$.

Define
$$
\cB_\LL := \{ C_\LL(Z) \mid Z \in \decirr^\sr(\LL) \}.
$$
We refer to the elements of $\cB_\LL$ as \emph{generic
Caldero-Chapoton functions}.

\begin{Prop}
Let $\LL = \CQ{Q}/I$ be a basic algebra.
If $\Ker(B_Q) \cap \Q_{\ge 0}^n = 0$, then
$\cB_\LL$ is 
linearly independent in $\cA_\LL$.
\end{Prop}

\begin{proof}
For each $Z \in \decirr^\sr(\LL)$ there is some $\cM \in Z$
such that $g_\LL(\cM) = g_\LL(Z)$ and $C_\LL(\cM) = C_\LL(Z)$.
By Theorem~\ref{parametrization}(i) the generic
$g$-vectors of the strongly reduced
components of decorated representations of $\LL$ are
pairwise different.
Now Proposition~\ref{li2} yields the result.
\end{proof}

If $\cB_\LL$ is a basis of $\cA_\LL$, then we call $\cB_\LL$
the \emph{generic basis} of $\cA_\LL$.

%%%%%%%%%%%%%%%%%%%%%%%%%
\subsection{$CC$-clusters}
%%%%%%%%%%%%%%%%%%%%%%%%%
For a component cluster $\cU$ of a basic algebra 
$\LL$
let
$$
\cC_\cU := \{ C_\LL(Z) \mid Z \in \cU \}
\text{\;\;\; and \;\;\;}
\cM_\cU := \{ \prod_{Z \in \cU} C_\LL(Z)^{a_Z} \mid a_Z \in I_Z \}
$$
where 
$$
I_Z := \begin{cases}
\N & \text{if $E_\LL(Z,Z) = 0$},\\
\{ 0,1 \} & \text{otherwise}.
\end{cases}
$$
(In each of the products above
we assume that $a_Z = 0$ for all but finitely many
$Z \in \cU$.)
The set $\cC_\cU$ is called a $CC$-\emph{cluster} of $\LL$, 
and the elements in $\cM_\cU$ are 
$CC$-\emph{cluster monomials}.
(The letters $CC$ just indicate that we deal with sets of
Caldero-Chapoton functions.)
A $CC$-cluster $\cC_\cU$ is $E$-\emph{rigid} provided
$E_\LL(Z) = 0$ for all $Z \in \cU$.

Note that 
$$
\cC_\cU \subseteq \cM_\cU \subseteq \cA_\LL.
$$
The following result is a direct consequence of the
definition of $\cB_\LL$ and 
Theorem~\ref{decomp2}.

\begin{Prop}
Let $\LL$ be a basic algebra.
Then
$$
\cB_\LL = \bigcup_{\cU} \cM_\cU
$$
where the union is over all component clusters $\cU$ of $\LL$.
\end{Prop}

%%%%%%%%%%%%%%%%%%%%%%%%%%%%
\subsection{A change of perspective}
%%%%%%%%%%%%%%%%%%%%%%%%%%%%
The $CC$-clusters are a generalization of the clusters
of a cluster algebra defined by Fomin and Zelevinsky.
In general, the Fomin-Zelevinsky cluster monomials
form just a small subset of the set of $CC$-cluster monomials.
Recall that 
the Fomin-Zelevinsky cluster monomials are obtained by the inductive
procedure of cluster mutation \cite{FZ1,FZ2}, and the 
relation between neighbouring clusters is described by
the exchange relations.
One can see the exchange relations as part of the definition
of a cluster algebra.
On the other hand,
the definition of a
Caldero-Chapoton algebra does not involve any mutations
of $CC$-clusters. 
The $CC$-clusters are given by collections
of irreducible components, and they do not have to be constructed inductively.
One can find a meaningful notion of \emph{neighbouring}
$CC$-clusters, and it remains quite a challenge to actually
determine an analogue of the exchange relations.

%%%%%%%%%%%%%%%%%%%
\subsection{Open problems}
%%%%%%%%%%%%%%%%%%%
In this section let $\LL$ be any basic algebra.
The following conjecture is again quite optimistic
in this generality.

\begin{Conj}\label{conj1}
$\cB_\LL$ is a $\C$-basis of $\cA_\LL$.
\end{Conj}

Conjecture~\ref{conj1} is true for every $\LL = \CQ{Q}$ with
$Q$ an acyclic quiver and also for numerous other examples,
see \cite{GLSChamber}.

\begin{Problem}\label{conj2}
Does
the set 
$$
\{ C_\LL(Z) \mid Z \in \decirr^\sr(\LL),\, 
E_\LL(Z) = 0 \}
$$
generate $\cA_\LL$ as a $\C$-algebra?
\end{Problem}

We say that 
$\cA_\LL$ has the \emph{Laurent phenomenon property}, if
for any 
$E$-rigid component cluster
$\{ Z_1,\ldots,Z_n \}$ of $\LL$,
we have
$$
\cA_\LL \subseteq \C[C_\LL(Z_1)^\pm,\ldots,C_\LL(Z_n)^\pm].
$$

\begin{Problem}\label{conj3}
Under which assumptions does $\cA_\LL$ have the Laurent phenomenon property?
\end{Problem}

%%%%%%%%%%%%%%%%%%%%%%%%%%%%%%%%%%%%%%%%%%%%%%
%%%%%%%%%%%%%%%%%%%%%%%%%%%%%%%%%%%%%%%%%%%%%%

\section{Caldero-Chapoton algebras and cluster algebras}\label{sec7}

%%%%%%%%%%%%%%%%%%%%%%%%%%%%%%%%%%%%%%%%%%%%%%
%%%%%%%%%%%%%%%%%%%%%%%%%%%%%%%%%%%%%%%%%%%%%%

%%%%%%%%%%%%%%%%%%%%%%%%%%%%%%%%%%%%%%%%%%%%%%%
\subsection{Caldero-Chapoton algebras of Jacobian algebras}
%%%%%%%%%%%%%%%%%%%%%%%%%%%%%%%%%%%%%%%%%%%%%%%
Suppose that $Q$ is a 2-acyclic 
quiver
with a non-degenerate potential $W$,
and let $\LL := \cP(Q,W)$ be the associated Jacobian algebra.
Let $\cA_Q$ and $\cA_Q^\up$ be the cluster algebra
and upper cluster algebra associated to $Q$, respectively.
Set $\cA_{Q,W} := \cA_\LL$, $\cB_{Q,W} := \cB_\LL$.
Let 
$$
\cM_{Q,W} := \{ C_\LL(Z) \mid Z \in \decirr^\sr(\LL),\,
E_\LL(Z) = 0
\}.
$$

The first part of the following proposition is a consequence of \cite[Lemma~5.2]{DWZ2},
compare also the calculation at the end of \cite[Section~6.3]{GLSChamber}.
The rest follows from \cite[Corollary~7.2]{DWZ2}.

\begin{Prop}
We have
$$
\cA_Q \subseteq \cA_{Q,W} \subseteq \cA_Q^\up. 
$$
The set $\cM_Q$ of cluster monomials of $\cA_Q$ is contained
in $\cB_{Q,W}$.
More precisely, we have
$$
\cM_Q \subseteq 
\cM_{Q,W} \subseteq
\cB_{Q,W}.
$$
\end{Prop}

In general, the sets $\cM_Q$, $\cM_{Q,W}$ and $\cB_{Q,W}$ are pairwise different.

%%%%%%%%%%%%%%%%
\subsection{Example}
%%%%%%%%%%%%%%%%
Let $Q$ be the quiver
$$
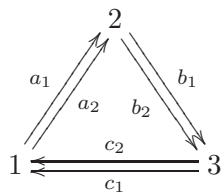
\xymatrix@-1.2pc{
&&2 \ar@<0.4ex>[dddrr]^{b_1}\ar@<-0.4ex>[dddrr]_{b_2}\\
&&&&\\
&&&&\\
1 \ar@<0.4ex>[rruuu]^{a_1}\ar@<-0.4ex>[rruuu]_{a_2} &&&&3 \ar@<0.4ex>[llll]^{c_1}\ar@<-0.4ex>[llll]_{c_2}
}
$$
and define
\begin{align*}
W_1 &:= c_1b_1a_1 + c_2b_2a_2,\\
W_2 &:= c_1b_1a_1 + c_2b_2a_2 - c_2b_1a_2c_1b_2a_1.
\end{align*}
It is not difficult to check that $\cP(Q,W_1)$ is infinite dimensional and $\cP(Q,W_2)$ is finite-dimensional.
By
\cite[Proposition~1.26]{BFZ} the algebras $\cA_Q$ and $\cA_Q^\up$ do not coincide.
The potentials $W_1$ and $W_2$ are both non-degenerate,
see \cite[Example~8.6]{DWZ1} and \cite[Example~8.2]{La2},
respectively. 
Furthermore, by
\cite[Example~4.3]{P2} 
the set $\cB_{Q,W_2}$ of generic functions 
is not contained in $\cA_Q$.
In particular, $\cA_Q$ and $\cA_{Q,W_2}$ do not coincide.
We conjecture that $\cA_Q = \cA_{Q,W_1}$ and
$\cA_Q^\up = \cA_{Q,W_2}$.

%%%%%%%%%%%%%%%%%%%%%
\subsection{Open problems}
%%%%%%%%%%%%%%%%%%%%%

\begin{Problem}\label{conj4}
Are there always 
non-degenerate potentials $W_1$ and
$W_2$ of $Q$ such that
$\cA_Q = \cA_{Q,W_1}$ and $\cA_Q^\up = \cA_{Q,W_2}$?
\end{Problem}

\begin{Problem}\label{conj5}
For a non-degenerate potential $W$ of $Q$ find a
necessary and sufficient condition for 
$\cA_{Q,W} = \cA_Q^\up$.
Is this related to the two conditions
\begin{itemize}

\item[(i)]
$\cP(Q,W)$ is finite-dimensional,

\item[(ii)]
$W$ is rigid (see \cite[Section~6]{DWZ1} for the definition)?

\end{itemize}
\end{Problem}

\begin{Problem}\label{conj45}
Suppose that there is only one non-degenerate potential $W$
of $Q$ up to right equivalence.
Does it follows that
$\cA_Q = \cA_Q^\up$?
\end{Problem}

%%%%%%%%%%%%%%%%%%%%%%%%%%%%%%%%%%%%%%%%%%%%%
%%%%%%%%%%%%%%%%%%%%%%%%%%%%%%%%%%%%%%%%%%%%%

\section{Sign-coherence of generic $g$-vectors}\label{sec8}

%%%%%%%%%%%%%%%%%%%%%%%%%%%%%%%%%%%%%%%%%%%%%
%%%%%%%%%%%%%%%%%%%%%%%%%%%%%%%%%%%%%%%%%%%%%

The following result implies Theorem~\ref{signcohintro}.
The special case, where $\LL = \cP(Q,W)$ is a Jacobian algebra
with non-degenerate potential $W$ and $\cU$ is an $E$-rigid
component cluster, is proved in
\cite[Theorem~3.7(1)]{P2}.

\begin{Thm}
Let $\LL$ be a basic algebra, and let
$\cU$ be a component cluster
of $\LL$.
Then the set $\{ g_\LL(Z) \mid Z \in \cU\}$ is sign-coherent.  
\end{Thm}

\begin{proof}
Assume that $\{ g_\LL(Z) \mid Z \in \cU\}$
is not sign-coherent.
Thus there are $Z_1,Z_2 \in \cU$ such that the set $\{ g_\LL(Z_1),g_\LL(Z_2) \}$ is not sign-coherent.
Since $\cU$ is a component cluster, we know from Theorem~\ref{decomp2} that $Z := \overline{Z_1 \oplus Z_2}$
is a strongly reduced component.
By Lemma~\ref{gadditive} we have $g_\LL(Z) = g_\LL(Z_1) + g_\LL(Z_2)$.
By Lemma~\ref{reduction}(ii)
there is some $p$ such that $Z,Z_1,Z_2 \in \decirr_{<p}^\sr(\LL_p)$.
We also know that $g_{\LL_p}(Z) = g_\LL(Z)$ and
$g_{\LL_p}(Z_i) = g_\LL(Z_i)$ for $i=1,2$, and that  
$$
I_0^{\LL_p}(Z) = I_0^{\LL_p}(Z_1) + I_0^{\LL_p}(Z_2)
\text{\;\;\; and \;\;\;} 
I_1^{\LL_p}(Z) = I_1^{\LL_p}(Z_1) + I_1^{\LL_p}(Z_2).
$$
For $i=1,2$ let $(\bd_i,\bv_i) := \dimv(Z_i)$.

We first assume that $\bv_1 = \bv_2 = 0$.
Since $\{ g_\LL(Z_1),g_\LL(Z_2)\}$ is not sign-coherent, 
we get from Lemma~\ref{lemma3.3} that
$$
\add(I_0^{\LL_p}(Z_1)) \cap \add(I_1^{\LL_p}(Z_2)) \not= 0
\text{\;\;\; or \;\;\;}
\add(I_1^{\LL_p}(Z_1)) \cap \add(I_0^{\LL_p}(Z_2)) \not= 0,
$$
a contradiction to Theorem~\ref{plamondon1}(ii).

Next, assume that $\bv_1$ and $\bv_2$ are both non-zero.
The components $Z_1$ and $Z_2$ are indecomposable.
It follows that $Z_1$ and $Z_2$ are just the orbits of
some negative simple representations.
But then $\{ g_\LL(Z_1),g_\LL(Z_2)\}$ has to be sign-coherent,
a contradiction.

Finally, let $\bv_1 = 0$ and $\bv_2 \not= 0$.
Thus we get $Z_2 = \cO(\cS_i^-)$ for some $1 \le i \le n$.
This implies $g_\LL(Z_2) = \be_i$.
Since $\{ g_\LL(Z_1),g_\LL(Z_2) \}$ is not sign-coherent,
the $i$th entry of $g_\LL(Z_1)$ has to be negative.
It follows that the socle of each representation
in $Z_1$ has $S_i$ as a composition factor.
In particular, the $i$th entry $d_i$ of $\bd_1$ is non-zero.
But we also have $E_\LL(Z_1,Z_2) = 0$.
Now Lemma~\ref{negsimple} implies that $d_i = 0$, a contradiction.
\end{proof}

%%%%%%%%%%%%%%%%%%%%%%%%%%%%%%%%%%%%%%%%%%%%%
%%%%%%%%%%%%%%%%%%%%%%%%%%%%%%%%%%%%%%%%%%%%%

\section{Examples}\label{sec9}

%%%%%%%%%%%%%%%%%%%%%%%%%%%%%%%%%%%%%%%%%%%%%
%%%%%%%%%%%%%%%%%%%%%%%%%%%%%%%%%%%%%%%%%%%%%

%%%%%%%%%%%%%%%%%%%%%%%%%%%%%%%%%%%%%%%%%%%%%%
\subsection{Strongly reduced components for hereditary algebras}\label{srhereditary}
%%%%%%%%%%%%%%%%%%%%%%%%%%%%%%%%%%%%%%%%%%%%%%

%%%%%%%%%%%%%
\subsubsection{}
%%%%%%%%%%%%%
Assume that $\LL = \CQ{Q}$ with $Q$ an acyclic quiver.
Thus $\LL$ is equal to the ordinary path
algebra $\C Q$.
Clearly, for each $(\bd,\bv) \in \N^n \times \N^n$
the variety $\decrep_{\bd,\bv}(\LL)$ is an affine space.
In particular, it has just one irreducible component,
namely $Z_{\bd,\bv} := \decrep_{\bd,\bv}(\LL)$.

\begin{Lem}\label{schofield1}
The following hold:
\begin{itemize}

\item[(i)]
For irreducible components 
$Z_{\bd_1,0},Z_{\bd_2,0} \in \decirr(\LL)$ we have 
$$
\gExt_\LL^1(Z_{\bd_1,0},Z_{\bd_2,0}) = 
E_\LL(Z_{\bd_1,0},Z_{\bd_2,0}).
$$

\item[(ii)]
$Z_{\bd,\bv}$ is strongly reduced
if and only if $d_iv_i = 0$ for all $1 \le i \le n$.

\end{itemize}
\end{Lem}

\begin{proof}
Since $\LL$ is a finite-dimensional hereditary algebra,
we have 
$$
\dim \Ext_\LL^1(M,N)  = 
\dim \Hom_\LL(\tau_\LL^-(N),M)
$$ 
for all $M,N \in \rep(\LL)$.
Now Proposition~\ref{HomologicalE2} implies (i).
In particular, for $Z = Z_{\bd,0}$ we have $e_\LL(Z) = E_\LL(Z)$.
Since $Z = \decrep_{\bd,0}(\LL)$ is an affine space, 
Voigt's Lemma implies that $c_\LL(Z) = e_\LL(Z)$.
Thus $Z$ is strongly reduced. 
The components
$Z_{0,\be_i}$ are obviously also strongly reduced.
Now Corollary~\ref{cor5.5} yields (ii).
\end{proof}

The following result is a direct consequence of Lemma~\ref{schofield1} and
Schofield's \cite{Scho} ground breaking work on general representations of acyclic quivers.
For all unexplained terminology we refer to \cite{Scho}.

\begin{Prop}\label{schofield2}
Let $\LL = \CQ{Q}$ with $Q$ an acyclic quiver.
Then
the indecomposable strongly reduced components
are the components
$Z_{\bd,0}$, where $\bd$ is a Schur root, and the components
$Z_{0,\be_1},\ldots,Z_{0,\be_n}$, where
$\be_i$ is the $i$th standard basis vector of $\Z^n$.
\end{Prop}

For a finite-dimensional path algebra $\LL = \C Q$ one can use 
Schofield's algorithm \cite{Scho} (see also \cite{DW} for a more efficient version of the algorithm) to determine the canonical decomposition of a dimension vector, and one can also use it to
decide if $\gExt_\LL^1(Z_1,Z_2)$ is zero or not.
So at least in principle, the graph $\Gamma(\decirr^\sr(\LL))$ can
be computed.
However, even in this case there are numerous interesting open questions on the structure of the graph 
$\Gamma(\decirr^\sr(\LL))$, see \cite{Sche}.

%%%%%%%%%%%%%%%%%%%%%%%%%%%%%%%%%%%%%%%%%
\subsection{Strongly reduced components for 
1-vertex algebras}
%%%%%%%%%%%%%%%%%%%%%%%%%%%%%%%%%%%%%%%%%

\begin{Prop}
Let $\LL = \CQ{Q}/I$ be a basic algebra with $n=1$.
Then the following hold:
\begin{itemize}

\item[(i)]
If $\overline{\LL}$ is finite-dimensional, then 
the indecomposable strongly reduced components
in $\decirr(\LL)$ are
$\cO(\cS_1^-)$ and the closure of $\cO((\overline{I}_1,0))$, where
$\overline{I}_1$ is 
the injective envelope of the simple $\overline{\LL}$-module $S_1$.

\item[(ii)]
If $\overline{\LL}$ is infinite-dimensional, then 
the only indecomposable strongly reduced component 
in $\decirr(\LL)$ is
$\cO(\cS_1^-)$.

\end{itemize}
\end{Prop}

\begin{proof}
Assume that $\overline{\LL}$ is
finite-dimensional.
Then Theorem~\ref{parametrization} implies that
$\Ima(G_\LL^\sr) = \Z$. 
For $m \ge 0$, we know
that the orbit closures of $(\cS_1^-)^m$ and $(\overline{I}_1,0)^m$
are $E$-rigid strongly reduced components with
generic $g$-vectors $m\be_1$ and $-m\be_1$, respectively.
This implies (i).
Part (ii) follows from the proof of Theorem~\ref{parametrization}(ii).
\end{proof}

%%%%%%%%%%%%%%%%%%%%%%%%%%%%%%%%%%%%%%
\subsection{Strongly reduced components for some
representation-finite algebras}
%%%%%%%%%%%%%%%%%%%%%%%%%%%%%%%%%%%%%%
By Corollary~\ref{cor6.5} each vertex of the component graph of a representation-finite basic algebra has a loop.
In the following examples,
for each $E$-rigid indecomposable
strongly reduced component, we just display the indecomposable decorated representation whose orbit closure
is the component.
We describe representations by
displaying their socle series and their composition factors.
For $1 \le i \le n$ we write $i$ and $-i$ instead of $\cS_i$ and $\cS_i^-$, respectively.
For a decorated representation of the form $\cM = (M,0)$
we just display $M$.

%%%%%%%%%%%%%
\subsubsection{}\label{9.3.1}
%%%%%%%%%%%%%
Let $Q$ be the quiver
$$
\xymatrix{
1 & 2 \ar[l]_b \ar@(ur,dr)[]^a
}
$$
and let $\LL := \CQ{Q}/I$, where $I$ is generated
by $ba$.
Then $\Gamma(\decirr^\sr(\LL))$ looks as follows:
$$
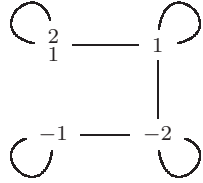
\xymatrix@-1.5pc{
&&\\
&{\bsm 2\\1 \esm} \ar@{-}[rr]\ar@{-}@(u,l)[] && {\bsm 1 \esm}\ar@{-}@(u,r)[]
\ar@{-}[dd]\\
\\
&{\bsm -1 \esm} \ar@{-}[rr] \ar@{-}@(l,d)[]&& {\bsm -2 \esm}
\ar@{-}@(r,d)[]\\
&&
}
$$
For $p=2$, the component graph 
$\Gamma(\decirr^\sr(\LL_p))$ looks as follows:
$$
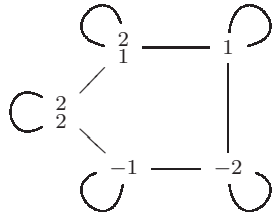
\xymatrix@-1.5pc{
&&\\
&{\bsm 2\\1 \esm}\ar@{-}@(u,l)[] \ar@{-}[rr] && {\bsm 1 \esm}\ar@{-}@(u,r)[]
\ar@{-}[dd]\\
{\bsm 2\\2 \esm}\ar@{-}@(ul,dl)[] \ar@{-}[ur]\ar@{-}[dr]\\
&{\bsm -1 \esm} \ar@{-}@(d,l)[]\ar@{-}[rr] && {\bsm -2 \esm}
\ar@{-}@(d,r)[]\\
&&
}
$$
To repair the somewhat non-symmetric graph
$\Gamma(\decirr^\sr(\LL))$ one could insert a vertex for the infinite-dimensional indecomposable injective $\LL$-module $I_2$.
Such aspects will be dealt with elsewhere.

%%%%%%%%%%%%%
\subsubsection{}
%%%%%%%%%%%%%
Let $Q$ be the quiver
$$
\xymatrix{
1 & 2 \ar[l]_b & 3 \ar[l]_a
}
$$
and let $\LL := \CQ{Q}/I$, where $I$ is generated
by $ba$.
Then $\Gamma(\decirr^\sr(\LL))$ looks as follows:
$$
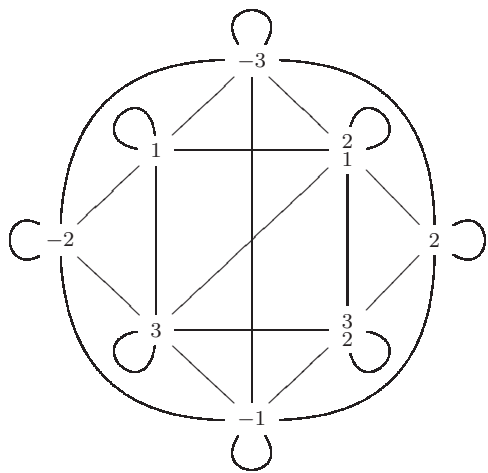
\xymatrix@-1.5pc{
&&\\
&&&&{\bsm -3 \esm} \ar@{-}[ddll] \ar@{-}[ddrr]
\ar@{-}@/^2.2pc/[ddddrrrr]\ar@{-}@/_2.2pc/[ddddllll]
\ar@{-}[dddddddd]\ar@{-}@(ul,ur)[]\\
&&\\
&& {\bsm 1 \esm}\ar@{-}[dddd] \ar@{-}@(u,l)[]\ar@{-}[rrrr]&&&&{\bsm 2\\1 \esm}\ar@{-}[dddd]\ar@{-}[ddddllll]\ar@{-}@(u,r)[]\\
&&\\
{\bsm -2 \esm} \ar@{-}@(ul,dl)[] \ar@{-}[uurr]\ar@{-}[ddrr]&&&&&&&&{\bsm 2 \esm}\ar@{-}[ddll]\ar@{-}[uull]\ar@{-}@(ur,dr)[] 
\\
&&\\
&& {\bsm 3 \esm}\ar@{-}@(d,l)[]\ar@{-}[rrrr] &&&&{\bsm 3\\2 \esm}
\ar@{-}@(d,r)[]\\
&&\\
&&&&{\bsm -1 \esm} \ar@{-}[uurr]\ar@{-}@(dl,dr)[]
\ar@{-}[uull] \ar@{-}@/_2.2pc/[uuuurrrr]\ar@{-}@/^2.2pc/[uuuullll]
\\
&&
}
$$
Note that for $M = \bsm 3\\2 \esm$ and 
$N = \bsm 1 \esm$ and we have $\Ext_\LL^1(M,N) = 0$  but $E_\LL((M,0),(N,0)) = \Hom_\LL(\tau^-(N),M) \not= 0$.

%%%%%%%%%%%%%%%%%%%%%%%%%%%%%%%%%%%%%
\subsection{Examples of Caldero-Chapoton algebras}
%%%%%%%%%%%%%%%%%%%%%%%%%%%%%%%%%%%%%

\subsubsection{}
Let $Q$ be the quiver
$$
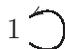
\xymatrix{
1 \ar@(dr,ur)
}
$$
and let $\LL := \CQ{Q}$.
The skew-symmetric adjacency matrix $B_Q$ of $Q$ is
just $(0)$.
Up to isomorphism, for each $d \ge 1$ there
is a unique indecomposable representation $M_d$ of $\LL$
with $\dim(M_d) = d$, given by a nilpotent Jordan block of size $d$.
One easily checks that
$$
C_\LL(M_d) = d+1.
$$
This implies
$\cA_\LL = \C$.

For $p \ge 2$ the indecomposable representations
of the $p$-truncation $\LL_p$ are
$M_1,\ldots,M_p$, and we get
$$
C_{\LL_p}(M_d) =
\begin{cases}
(d+1)x_1^{-1} & \text{if $d=p$},\\
(d+1) & \text{otherwise}.
\end{cases}
$$
This implies
$\cA_{\LL_p} = \C[x_1^{-1}]$.

%%%%%%%%%%%%%%%%%%%%%%%%%%%%%%%%%%%%%%%%%%%
\subsubsection{}\label{a3example}
%%%%%%%%%%%%%%%%%%%%%%%%%%%%%%%%%%%%%%%%%%%
In this section, 
let $Q$ be the quiver 
$$
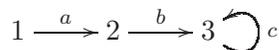
\xymatrix{1\ar^a[r]&2\ar^b[r]&3\ar@(dr,ur)_c}
$$
and let $\LL := \CQ{Q}/I$, where $I$ is the ideal generated by
$c^2$. 
(This example is
closely related to Zhou and Zhu's \cite{ZZ} study of cluster tubes.
We thank the anonymous referee for pointing this out.
The considerations in this section 
also hold for the corresponding more general case of a linear quiver with $n$
vertices and a loop $c$ with $c^2 = 0$. 
This and also a comparison with the results in \cite{ZZ} will be carried out in a
separate publication.)

The basic algebra $\LL$ is a representation-finite string algebra, and its Auslander-Reiten quiver looks as follows:
(Recall that there is an arrow from $M$ to $N$ if and only if there
is an irreducible homomorphism from $M$ to $N$,
and for all non-projective $M$ we draw a dashed arrow
from $M$ to its Auslander-Reiten translate $\tau_\LL(M)$. )
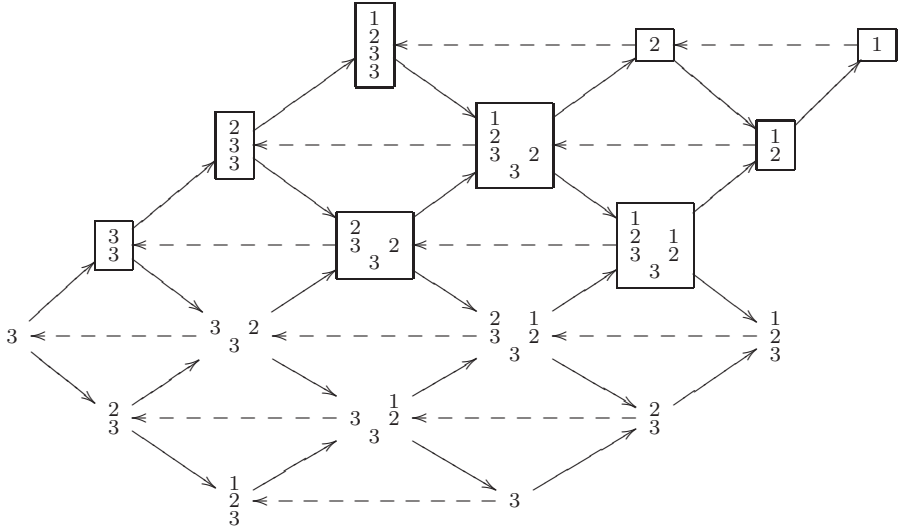
$$
\xymatrix@R=6pt{
&&&*+[F]{\bsm 1\\2\\3\\3 \esm}
\ar[dr]
&&*+[F]{\bsm 2 \esm}\ar[dr]\ar@{-->}[ll]&&*+[F]{\bsm 1 \esm}\ar@{-->}[ll]\\
&&*+[F]{\bsm 2\\3\\3 \esm}\ar[ur]\ar[dr]&&*+[F]{\bsm 1 && \\2 &&\\3 && 2\\& 3 \esm}\ar[ur]\ar[dr]\ar@{-->}[ll]&&*+[F]{\bsm 1\\2 \esm}\ar[ur]\ar@{-->}[ll]&&\\
&*+[F]{\bsm 3\\3 \esm}\ar[dr]\ar[ur]&&*+[F]{\bsm 2&&\\3&&2\\&3 \esm}\ar@{-->}[ll]\ar[dr]\ar[ur]&&*+[F]{\bsm 1&&\\2&&1\\3&&2\\&3 \esm}\ar@{-->}[ll]\ar[ur]\ar[dr]&&
\\
*+{\bsm 3 \esm}\ar[ur]\ar[dr]&&*+{\bsm 3&&2\\&3 \esm}\ar[dr]\ar[ur]\ar@{-->}[ll]&&*+{\bsm 2&&1\\3&&2\\&3 \esm}\ar[ur]\ar[dr]\ar@{-->}[ll]&&*+{\bsm 1\\2\\3 \esm}\ar@{-->}[ll]
\\
&*+{\bsm 2\\3 \esm}\ar[dr]\ar[ur]&&*+{\bsm &&1\\3&&2\\&3 \esm}\ar@{-->}[ll]\ar[dr]\ar[ur]&&*+{\bsm 2\\3 \esm}\ar[ur]\ar@{-->}[ll]&&\\
&&*+{\bsm 1\\2\\3 \esm}\ar[ur]&&*+{\bsm 3 \esm}\ar[ur]\ar@{-->}[ll]&&&&
}
$$ 
In this quiver the two south-west  and the two south-east edges are identified. 
The framed representations are the indecomposable $E$-rigid 
representations of $\LL$.
We have
\begin{align*}
I_1 &= {\bsm 1 \esm}, &
I_2 &= {\bsm 1\\2 \esm}, &
I_3 &= {\bsm 1\\2&&1\\3&&2\\&3\esm}.
\end{align*}

We now describe explicitly the Caldero-Chapoton functions associated to the 12 indecomposable $E$-rigid decorated representations of $\LL$. 
By definition $C_\Lambda (\cS_i^-)=x_i$, for $i=1,2,3$. 
The remaining 9 functions are 
\begin{align*}
C_\Lambda({\bsm 1 \esm}) &= {\dfrac{1+x_2}{x_1}}, &
C_\Lambda\left({\bsm 1\\2\esm}\right)&={\dfrac{x_1+x_3+x_2x_3}{x_1x_2}},\\
C_\Lambda\left({\bsm 1\\2\\3\\3\esm}\right)&={\dfrac{x_1x_2^2+x_1x_2+x_1+x_3+x_2x_3}{x_1x_2x_3}}, &
C_\Lambda({\bsm 2 \esm})&={\dfrac{x_1+x_3}{x_2}},\\
C_\Lambda\left({\bsm 3\\3\esm}\right)&={\dfrac{x_2^2+x_2+1}{x_3}},&
C_\Lambda\left({\bsm 2\\3\\3\esm}\right)&=
{\dfrac{x_1x_2^2+x_1x_2+x_1+x_3}{x_2x_3}},
\end{align*}
\begin{align*}
C_\Lambda\left({\bsm 1\\2\\3&&2\\&3\esm}\right)&=
\dfrac{x_1^2x_2^2+x_1^2x_2+x_1x_2x_3+2x_1x_3+x_1^2+x_1x_2x_3+x_2x_3^2+x_3^2}{x_1x_2^2x_3},\\
C_\Lambda\left({\bsm 1\\2&&1\\3&&2\\&3\esm}\right)&=\dfrac{x_1^2x_2^2+x_1^2x_2+x_1^2+x_1x_2x_3+2x_1x_3+x_3^2+x_1x_2^2x_3+2x_1x_2x_3+2x_2x_3^2+x_2^2x_3^2}{x_1^2x_2^2x_3},\\
C_\Lambda\left({\bsm 2\\3&&2\\&3\esm}\right)&=
\dfrac{x_1^2x_2^2+x_1^2x_2+x_1^2+x_1x_2x_3+2x_1x_3+x_3^2}{x_2^2x_3}.
\end{align*}
The Caldero-Chapoton functions associated to the 6 indecomposable non-$E$-rigid 
representations of $\Lambda$ are
\begin{align*}
C_\Lambda({\bsm 3\esm})&=x_2+1,\\
C_\Lambda\left({\bsm 2\\3\esm}\right)&=\dfrac{x_1x_2+x_1+x_3}{x_2},\\
C_\Lambda\left({\bsm 3&&2\\&3\esm}\right)&=\dfrac{x_1x_2^2+x_1x_2+x_1+x_3+x_2x_3}{x_2x_3},\\
C_\Lambda\left({\bsm 1\\2\\3\esm}\right)&=\dfrac{x_1x_2+x_1+x_3+x_2x_3}{x_1x_2},\\
C_\Lambda\left({\bsm &&1\\3&&2\\&3\esm}\right)&=\dfrac{x_1x_2^2+x_1x_2+x_1+x_2x_3+x_2^2x_3+x_2x_3+x_3}{x_1x_2x_3},\\
C_\Lambda\left({\bsm 2&&1\\3&&2\\&3\esm}\right)&=\dfrac{x_1^2x_2^2+x_1^2x_2+x_1^2+x_1x_2x_3+2x_1x_3+x_3^2+x_1x_2^2x_3+x_1x_2x_3+x_2x_3^2}{x_1x_2^2x_3}.
\end{align*}
The following statement says that in our example, 
there is a positive answer to Problem~\ref{conj2}.

\begin{Prop}\label{claim9.4}
The set
\[
\{ C_\LL(Z) \mid Z \in \decirr^\sr(\LL),\, 
E_\LL(Z) = 0 \}
\]
generates the Caldero-Chapoton algebra $\cA_\Lambda$ as a $\C$-algebra. 
\end{Prop}

\begin{proof}
It is enough to express the Caldero-Chapoton functions of the 6 indecomposable non-$E$-rigid representations in terms of
Caldero-Chapoton functions of 
$E$-rigid decorated representations.
An easy calculation yields
\begin{align*}
C_\Lambda({\bsm 3\esm})&=x_2+1,\\
C_\Lambda\left({\bsm 2\\3\esm }\right)&=x_1+C_\Lambda({\bsm 2 \esm}),\\
C_\Lambda\left({\bsm 3&&2\\&3\esm}\right)&=C_\Lambda\left({\bsm2\\3\\3\esm}\right)+1\\
C_\Lambda\left({\bsm1\\2\\3\esm}\right)&=
C_\Lambda\left({\bsm 1\\2 \esm}\right)+1,\\
C_\Lambda\left({\bsm &&1\\3&&2\\&3\esm}\right)&=C_\Lambda\left({\bsm 1\\2\\3\\3\esm}\right)+C_\Lambda({\bsm 1 \esm}),\\
C_\Lambda\left({\bsm 2&&1\\3&&2\\&3\esm}\right)&=C_\Lambda\left({\bsm 1\\2\\3&&2\\&3\esm}\right)+1.
\end{align*}
All summands of the right hand side of the above equations
are Caldero-Chapoton functions of $E$-rigid decorated
representations, and the 6 Caldero-Chapoton functions
of the indecomposable non-$E$-rigid representations are
on the left side.
(Note that $x_i = C_\LL(\cS_i^-)$ and $1 = C_\LL(0)$, and
$\cS_i^-$ and $0$ are both $E$-rigid.)
This finishes the proof.
\end{proof}

Since $\LL$ is representation-finite, 
each strongly reduced component contains an $E$-rigid decorated representation.
Each vertex of $\Gamma(\decirr^\sr(\LL))$ has a loop.
Let $\Gamma(\decirr^\sr(\LL))^\circ$ be the graph obtained
by deleting these loops. 
We display $\Gamma(\decirr^\sr(\LL))^\circ$ in the following picture. 
Note that each component cluster is $E$-rigid and contains exactly three
irreducible components.
$$
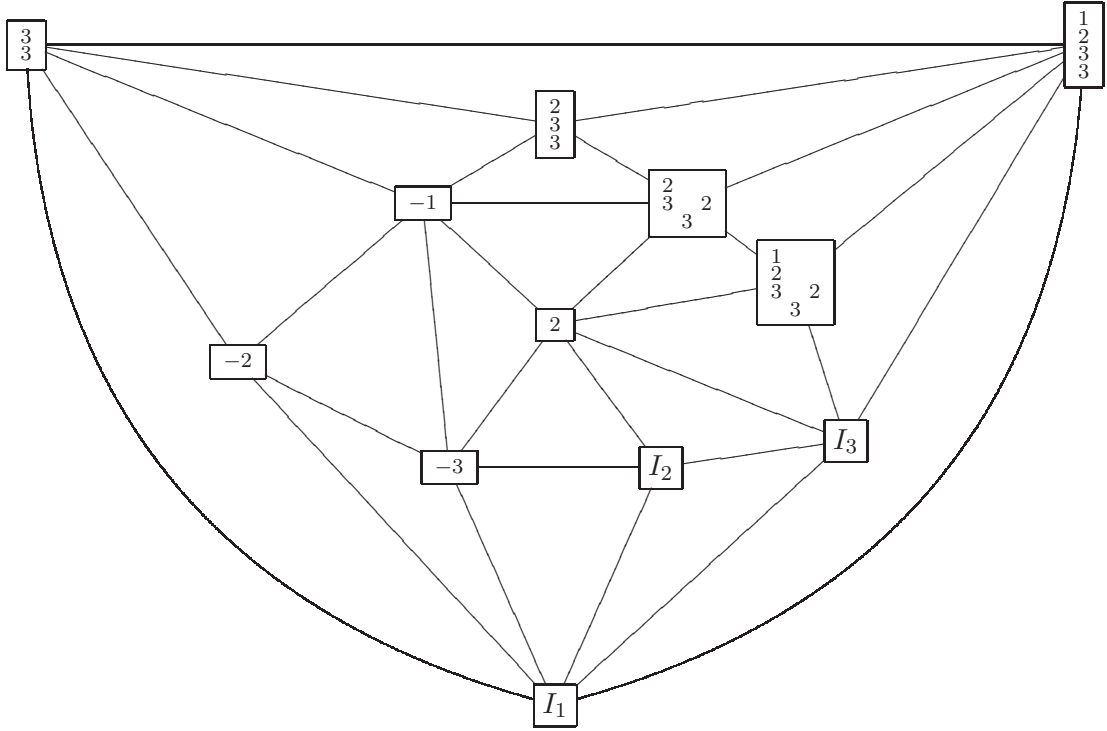
\begin{xy} 0;<1pt,0pt>:<0pt,-1pt>:: 
(0,0) *+[F]{\bsm 3\\3 \esm} ="0",
(200,30) *+[F]{\bsm 2\\3\\3 \esm} ="1",
(150,60) *+[F]{\bsm -1 \esm} ="2",
(250,60) *+[F][F]{\bsm 2&&\\3&&2\\&3 \esm} ="3",
(200,106) *+[F]{\bsm 2 \esm} ="4",
(160,160) *+[F]{\bsm -3 \esm} ="5",
(80,120) *+[F]{\bsm -2 \esm} ="6",
(240,160) *+[F]{I_2} ="7",
(310,150) *+[F]{I_3} ="8",
(291,90) *+[F][F]{\bsm 1\\2\\3&&2\\&3 \esm} ="9",
(400,0) *+[F]{\bsm 1\\2\\3\\3 \esm} ="10",
(200,250) *+[F]{I_1} ="11",
"1", {\ar@{-}"0"},
"0", {\ar@{-}"2"},
"0", {\ar@{-}"6"},
"0", {\ar@{-}"10"},
"0", {\ar@{-}@/_5pc/"11"},
"2", {\ar@{-}"1"},
"3", {\ar@{-}"1"},
"1", {\ar@{-}"10"},
"2", {\ar@{-}"3"},
"2", {\ar@{-}"4"},
"2", {\ar@{-}"5"},
"2", {\ar@{-}"6"},
"4", {\ar@{-}"3"},
"9", {\ar@{-}"3"},
"3", {\ar@{-}"10"},
"5", {\ar@{-}"4"},
"7", {\ar@{-}"4"},
"4", {\ar@{-}"8"},
"4", {\ar@{-}"9"},
"6", {\ar@{-}"5"},
"5", {\ar@{-}"7"},
"5", {\ar@{-}"11"},
"6", {\ar@{-}"11"},
"7", {\ar@{-}"8"},
"7", {\ar@{-}"11"},
"8", {\ar@{-}"9"},
"8", {\ar@{-}"10"},
"8", {\ar@{-}"11"},
"9", {\ar@{-}"10"},
"11", {\ar@{-}@/_5pc/"10"},
\end{xy}
$$

%%%%%%%%%%%%%
\subsubsection{}\label{sec9.4.3}
%%%%%%%%%%%%%
Let $Q$ be the 2-Kronecker quiver
$$
\xymatrix@-0.5pc{
1 \ar@<0.4ex>[d]\ar@<-0.4ex>[d]\\
2
}
$$
and let $\LL = \CQ{Q}$.
The following picture describes the quiver $\Gamma(\decirr^\sr(\LL))$. 
(For indecomposable strongly reduced components of the form $Z_{\bd,0}$ or
$Z_{0,\be_i}$ we just
display the vectors $\bd$ or $-\be_i$, respectively.)
$$
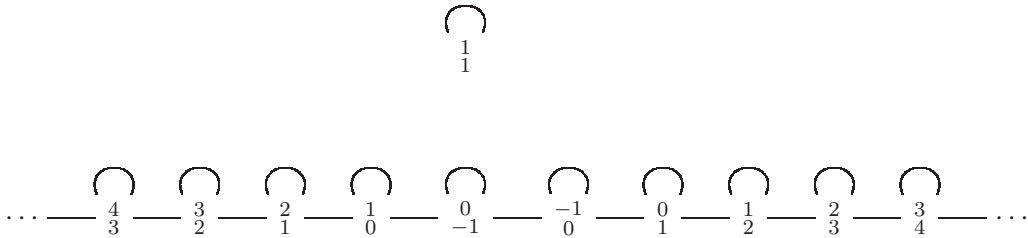
\xymatrix@-0.5pc{
&&\\
&&&&& {\bsm1\\1\esm} \ar@{-}@(ul,ur)[]\\
&&\\
\cdots \ar@{-}[r] & 
{\bsm 4\\3\esm} \ar@{-}[r]\ar@{-}@(ul,ur)[] &
{\bsm 3\\2\esm} \ar@{-}[r]\ar@{-}@(ul,ur)[] &
{\bsm 2\\1\esm} \ar@{-}[r]\ar@{-}@(ul,ur)[] &
{\bsm 1\\0\esm} \ar@{-}[r]\ar@{-}@(ul,ur)[] &
{\bsm 0\\-1\esm} \ar@{-}[r]\ar@{-}@(ul,ur)[] &
{\bsm -1\\0\esm} \ar@{-}[r]\ar@{-}@(ul,ur)[] &
{\bsm 0\\1\esm} \ar@{-}[r]\ar@{-}@(ul,ur)[] &
{\bsm 1\\2\esm} \ar@{-}[r]\ar@{-}@(ul,ur)[] & 
{\bsm 2\\3\esm} \ar@{-}[r]\ar@{-}@(ul,ur)[] &
{\bsm 3\\4\esm} \ar@{-}[r]\ar@{-}@(ul,ur)[] & 
\cdots
}
$$
Thus there is exactly one component cluster $\{ Z \}$ 
of cardinality one,
and there are infinitely many component clusters of cardinality two.
One can easily check that $E_\LL(Z,Z) = 0$, hence the loop
at $Z$, but $E_\LL(Z) \not= 0$.
Thus $\{ Z \}$ is not $E$-rigid.
The other component clusters are $E$-rigid.
The $CC$-cluster monomials are
$$
C_\LL(\bsm 0\\-1\esm)^aC_\LL(\bsm -1\\0\esm)^b, \;\;\;\;
C_\LL(\bsm i+1\\i\esm)^aC_\LL(\bsm i\\i-1\esm)^b, \;\;\;\;
C_\LL(\bsm i-1\\i\esm)^aC_\LL(\bsm i\\i+1\esm)^b \;\;\;\;
\text{and} \;\;\;\;
C_\LL(\bsm 1\\1\esm)^a
$$
where $a,b,i \ge 0$.

The set $\cB_\LL$ of generic Caldero-Chapoton functions
is just the set of $CC$-cluster monomials.
Recall from \cite{BFZ} that for any acyclic quiver $Q$ we have
$\cA_Q = \cA_Q^\up$.
In this case, $\cB_\LL$ 
is a $\C$-basis of $\cA_Q$, see \cite{GLSChamber}.

For acyclic quivers $Q$ of wild representation type and
$\LL = \CQ{Q}$,
the component graph $\Gamma(\decirr^\sr(\LL))$ 
will have vertices without loops.
For example, let $Q$ be the 3-Kronecker quiver
$$
\xymatrix@-0.5pc{
1 \ar@<0.8ex>[d]\ar@<-0.8ex>[d]\ar[d]\\
2
}
$$
and let $\LL = \CQ{Q}$.
Let
$$
\phi = \left(\bbm -1&3\\-3&8\ebm\right)
$$
be the Coxeter matrix of $\LL$.
For $k \ge 0$ define
\begin{align*}
\bp_{2k} := \phi^k\left(\bsm 0\\1\esm\right),\\
\bp_{2k+1} := \phi^k\left(\bsm 1\\3\esm\right),\\
\bq_{2k} := \phi^{-k}\left(\bsm 1\\0\esm\right),\\
\bq_{2k+1} := \phi^{-k}\left(\bsm 3\\1\esm\right).
\end{align*}
Set $\bp_{-1} := -\be_1$ and $\bq_{-1} := -\be_2$. 
One connected component of the component graph
$\Gamma(\decirr^\sr(\LL))$
looks as follows:
$$
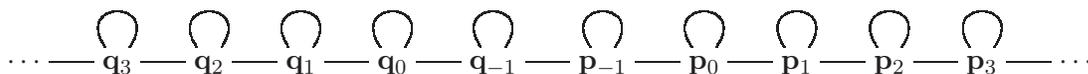
\xymatrix@-0.6pc{
&&\\
\cdots \ar@{-}[r] & 
\bq_3 \ar@{-}[r]\ar@{-}@(ul,ur)[] &
\bq_2 \ar@{-}[r]\ar@{-}@(ul,ur)[] &
\bq_1 \ar@{-}[r]\ar@{-}@(ul,ur)[] &
\bq_0 \ar@{-}[r]\ar@{-}@(ul,ur)[] &
\bq_{-1} \ar@{-}[r]\ar@{-}@(ul,ur)[] &
\bp_{-1} \ar@{-}[r]\ar@{-}@(ul,ur)[] &
\bp_0 \ar@{-}[r]\ar@{-}@(ul,ur)[] &
\bp_1 \ar@{-}[r]\ar@{-}@(ul,ur)[] & 
\bp_2 \ar@{-}[r]\ar@{-}@(ul,ur)[] &
\bp_3 \ar@{-}[r]\ar@{-}@(ul,ur)[] & 
\cdots
}
$$
These are precisely the $E$-rigid vertices of $\Gamma(\decirr^\sr(\LL))$.

The set of Schur roots of $Q$ consists of \emph{real} and \emph{imaginary}
Schur roots.
The above picture shows the real Schur roots (and the vectors
$-\be_1$ and $-\be_2$).  
The set $R_{\rm im}^+$ of imaginary Schur roots 
consists of all dimension vectors
$\bd = (d_1,d_2) \in \N^2$ with $d_2 \not= 0$
such that
$$
(3 - \sqrt{5})/2 \le d_1/d_2 \le (3 + \sqrt{5})/2,
$$
see \cite[Section~3]{DW} and \cite[Section~6]{K}.
There is no edge between
$Z_{\bd,0}$ and any other vertex of $\Gamma(\decirr^\sr(\LL))$.
In particular, there is no loop at $Z_{\bd,0}$.

The $CC$-cluster monomials are
$$
C_\LL(\bq_{-1})^aC_\LL(\bp_{-1})^b, \;\;\;\;
C_\LL(\bp_{i-1})^aC_\LL(\bp_i)^b, \;\;\;\;
C_\LL(\bq_i)^aC_\LL(\bq_{i-1})^b \;\;\;\;
\text{and} \;\;\;\;
C_\LL(\bd)
$$
where $a,b,i \ge 0$ and $\bd \in R_{\rm im}^+$.
Again it follows from \cite{GLSChamber} that these $CC$-cluster
monomials form a $\C$-basis
of $\cA_Q$.
It remains a challenge to compute the exchange relations
between all \emph{neighbouring} $CC$-clusters.
For the $E$-rigid $CC$-clusters, the exchange relations are known
from the Fomin-Zelevinsky exchange relations arising from
mutations of clusters.
But for $\bd,\bd_1,\bd_2 \in R_{\rm im}^+$ and $i \ge -1$
it remains
an open problem to express the products
$$
C_\LL(\bd)C_\LL(\bp_i), \;\;\;\;
C_\LL(\bd)C_\LL(\bq_i) \;\;\;\; and \;\;\;\;
C_\LL(\bd_1)C_\LL(\bd_2)
$$
as linear combinations of elements from the basis $\cB_\LL$.

\bigskip
%%%%%%%%%%%%%%%%%%%%%%%%%%%%%%%%%%%%%%%
{\parindent0cm \bf Acknowledgements.}\,
%%%%%%%%%%%%%%%%%%%%%%%%%%%%%%%%%%%%%%%
We thank Ben Davison, Christof Gei{\ss}, Sven Meinhardt and Pierre-Guy Plamondon for helpful discussions.
We are indebted to Charlotte Ricke for carefully reading the first version
of this article and for pointing out some inaccuracies.

%%%%%%%%%%%%%%%%%%%%%%%%%%%%%%%%%%%%%%%%%%%%%%%%%%%%%%%%%


\begin{thebibliography}{999}

%%%%%%%%%%%%%%%%%%%%%%%%%%%%%%%%%%%%%%%%%%%%%%%%%%%%%%%%%

\bibitem[A]{A}
C. Amiot,
\emph{Cluster categories for algebras of global dimension $2$ 
and quivers with potential}, 
Ann. Inst. Fourier (Grenoble) 59 (2009), no. 6, 2525--2590.

\bibitem[AR]{AR}
M. Auslander, I. Reiten,
\emph{Modules determined by their composition factors},
Illinois J. Math. 29 (1985), no. 2, 280--301. 

\bibitem[ARS]{ARS}
M. Auslander, I. Reiten, S. Smal\o,
\emph{Representation theory of Artin algebras}, 
Corrected reprint of the 1995 original. Cambridge Studies in Advanced Mathematics, 36. Cambridge University Press, Cambridge, 1997. 
xiv+425 pp.

\bibitem[ASS]{ASS}
I. Assem, D. Simson, A. Skowro\'nski,
\emph{Elements of the representation theory of associative algebras}, 
Vol. 1. Techniques of representation theory. London Mathematical Society Student Texts, 65. Cambridge University Press, Cambridge, 2006. x+458 pp. 

\bibitem[BFZ]{BFZ}
A. Berenstein, S. Fomin, A. Zelevinsky, 
\emph{Cluster algebras. III. Upper bounds and double Bruhat cells}, 
Duke Math. J. 126 (2005), no. 1, 1--52. 

\bibitem[CC]{CC}
P. Caldero, F. Chapoton, 
\emph{Cluster algebras as Hall algebras of quiver representations}, 
Comment. Math. Helv. 81 (2006), no. 3, 595--616.

\bibitem[C]{C}
G. Cerulli Irelli,
\emph{Cluster Algebras of Type $A_2^{(1)}$}, 
Algebr. Represent. Theory 15 (2012), no. 5, 977--1021. 

Algebras and Representation Theory (to appear),
40pp.,	arXiv:0904.2543

\bibitem[CBS]{CBS}
W. Crawley-Boevey,  J. Schr\"oer, 
\emph{Irreducible components of varieties of modules}, 
J. Reine Angew. Math. 553 (2002), 201--220. 

\bibitem[DW]{DW}
H. Derksen, J. Weyman,
\emph{On the canonical decomposition of quiver representations}, 
Compositio Math. 133 (2002), no. 3, 245--265.

\bibitem[DWZ1]{DWZ1}
H. Derksen, J. Weyman, A. Zelevinsky,
\emph{Quivers with potentials and their representations. I. Mutations}, 
Selecta Math. (N.S.) 14 (2008), no. 1, 59--119. 

\bibitem[DWZ2]{DWZ2}
H. Derksen, J. Weyman, A. Zelevinsky,
\emph{Quivers with potentials and their representations II: applications to cluster algebras}, 
J. Amer. Math. Soc. 23 (2010), no. 3, 749--790.

\bibitem[FZ1]{FZ1}
S. Fomin, A. Zelevinsky, 
\emph{Cluster algebras. I. Foundations}, 
J. Amer. Math. Soc. 15 (2002), no. 2, 497--529.

\bibitem[FZ2]{FZ2}
S. Fomin, A. Zelevinsky, 
\emph{Cluster algebras. II. Finite type classification},
Invent. Math. 154 (2003), no. 1, 63--121. 

\bibitem[FK]{FK}
C. Fu, B. Keller, 
\emph{On cluster algebras with coefficients and $2$-Calabi-Yau categories}, 
Trans. Amer. Math. Soc. 362 (2010), no. 2, 859--895. 

\bibitem[G]{G}
P. Gabriel, 
\emph{Finite representation type is open}, 
Proceedings of the International Conference on Rep- resentations of Algebras (Carleton Univ., Ottawa, Ont., 1974), Paper No. 10, 23pp. Carleton Math. Lecture Notes, No. 9, Carleton Univ., Ottawa, Ont., 1974.

\bibitem[GLaS]{GLaS}
C. Gei{\ss}, D. Labardini-Fragoso, J. Schr\"oer,
Preprint in preparation.

\bibitem[GLS]{GLSChamber}
C. Gei{\ss}, B. Leclerc, J. Schr\"oer,
\emph{Generic bases for cluster algebras and the Chamber ansatz}, 
J. Amer. Math. Soc. 25 (2012), no. 1, 21--76.

\bibitem[GS]{GS}
C. Gei{\ss}, J. Schr\"oer,
\emph{Extension-orthogonal components of preprojective varieties},
Trans. Amer. Math. Soc. 357 (2005), no. 5, 1953--1962. 

\bibitem[IR]{IR}
O. Iyama, I. Reiten,
$\tau$-\emph{tilting modules},
Talk in Trondheim on 28.03.2012.

\bibitem[K]{K}
V. Kac,
\emph{Infinite root systems, representations of graphs and invariant theory. II}, 
J. Algebra 78 (1982), no. 1, 141--162.

\bibitem[La1]{La1}
D. Labardini-Fragoso, 
\emph{Quivers with potentials associated to triangulated surfaces}, Proc. Lond. Math. Soc. (3) 98 (2009), no. 3, 797--839.

\bibitem[La2]{La2}
D. Labardini-Fragoso,
\emph{Quivers with potentials associated to triangulated surfaces, Part II: Arc representations}, 
Preprint (2009), 52pp., arXiv:0909.4100v2 

\bibitem[MRZ]{MRZ}
R. Marsh, M. Reineke, A. Zelevinsky, 
\emph{Generalized associahedra via quiver representations}, 
Trans. Amer. Math. Soc. 355 (2003), no. 10, 4171--4186.

\bibitem[Pa]{Pa}
Y. Palu,
\emph{Cluster characters for $2$-Calabi-Yau triangulated categories}, 
Ann. Inst. Fourier (Grenoble) 58 (2008), no. 6, 2221--2248.

\bibitem[P1]{P1}
P.-G. Plamondon, 
\emph{Cluster algebras via cluster categories with infinite-dimensional morphism spaces}, 
Compos. Math. 147 (2011), 1921--1954.

\bibitem[P2]{P2}
P.-G. Plamondon,
\emph{Generic bases for cluster algebras from the cluster category},
Int. Math. Res. Notices (to appear),
35pp., arXiv:1111.4431

\bibitem[Sche]{Sche}
S. Scherotzke,
\emph{Generalized clusters for acyclic quivers},
Preprint (2012).

\bibitem[Scho]{Scho}
A. Schofield,
\emph{General representations of quivers}, 
Proc. London Math. Soc. (3) 65 (1992), no. 1, 46--64.

\bibitem[ZZ]{ZZ}
Y. Zhou, B. Zhu, 
\emph{Cluster algebras of type C via cluster tubes},
Preprint (2010), 20pp., 
arXiv:1008.3444
\end{thebibliography}
\end{document}